\documentclass{article}[12pt]
\usepackage{url, graphicx, xcolor, hyperref, geometry}
\usepackage{latexsym,amsbsy}
\usepackage{lipsum}
\usepackage{amsfonts}
\usepackage{graphicx}
\usepackage{epstopdf}
\usepackage{algorithmic}
\usepackage{tikz}
\usetikzlibrary{matrix}
\usetikzlibrary{shapes,patterns,arrows,snakes,decorations.shapes}
\usetikzlibrary{positioning,fit,calc}
\usetikzlibrary{plotmarks}
\usepackage{amssymb,amsmath,amsthm}
\usepackage{smartdiagram}
\usepackage{mathabx}
\usepackage{float}
\usepackage{enumerate}
\usepackage{mnsymbol}

\newtheorem{theorem}{Theorem}[section]
\newtheorem{lemma}{Lemma}[section]

\newtheorem{definition}{Definition}[section]

\hypersetup{
    colorlinks,
    linkcolor={blue},
    citecolor={blue},
    urlcolor={blue}
}

\usepackage[figurename=Fig.]{caption}
 
\begin{document}

\title{Parameter-robust Braess-Sarazin-type  smoothers for  linear elasticity problems}
\author{Yunhui He\thanks{Department of Computer Science, The University of British Columbia, Vancouver, BC, V6T 1Z4, Canada,   \tt{yunhui.he@ubc.ca}.}\quad \quad \quad  Yu Li\thanks{Coordinated Innovation Center for Computable Modeling in Management Science, 
Tianjin University of Finance and Economics, No. 25 Zhujiang Rd., Hexi District, Tianjin 300222, China. This work is supported in part by the Scientific Research Plan of Tianjin Municipal Education Committee (2017KJ236).
\tt{liyu@tjufe,edu.cn}.}}

\maketitle
 \begin{abstract}
 In this work, we propose three  Braess-Sarazin-type multigrid relaxation schemes for solving linear elasticity problems, where the marker and cell scheme, a finite difference method, is used for the discretization.  The three relaxation schemes are Jacobi-Braess-Sarazin, Mass-Braess-Sarazin, and Vanka-Braess-Sarazin.   A local Fourier analysis (LFA) for the block-structured relaxation schemes is presented to study multigrid convergence behavior.   From LFA, we derive  optimal  LFA smoothing factor for each case. We obtain highly efficient smoothing factors, which are independent of Lam\'{e} constants. Vanka-Braess-Sarazin relaxation scheme leads to the most efficient one. In each relaxation, a Schur complement system needs to be solved. Due to the fact that direct solve is often expensive, an inexact version is developed,  where we simply use at most three weighted Jacobi iterations on the Schur complement system.  Finally, two-grid and V-cycle multigrid performances are presented to validate our theoretical results. Our numerical results show that inexact versions can achieve the same performance as that of exact versions and our methods are robust to the Lam\'{e} constants.
 
\end{abstract}

 
                   
 
 \section{Introduction}

We are interested in designing fast multigrid methods of numerical solution of the  following linear elasticity problem with homogeneous boundary condition  
 
\begin{eqnarray}
  - \epsilon \Delta\boldsymbol{u} -(\lambda+ \epsilon) \nabla (\nabla\cdot\boldsymbol{u}) &=&\boldsymbol{f}, \quad \text{in}\,\, \Omega \label{eq:elasticity-one}\\
 \boldsymbol{u}&=&0,\quad \text{on}\,\, \partial \Omega \label{eq:elasticity-bd}
\end{eqnarray}
where $\boldsymbol{u}=\begin{pmatrix} u \\ v \end{pmatrix}$ is the displacement, $\boldsymbol{f}=\begin{pmatrix} f^1 \\ f^2 \end{pmatrix}$ is the source term, and $\lambda=\frac{2\nu}{1-2\nu},\,  \epsilon >0$ are Lam\'{e} constants, in which $\nu$ is  the Poisson ratio.

When the material is nearly incompressible, i.e., the Poisson ratio  $\nu\rightarrow \frac{1}{2}$  (or $\lambda \rightarrow \infty$), the approximate solutions do not converge uniformly about $\lambda$. This phenomenon is called ‘‘Poisson locking’’. The popular way to  avoid the locking effects is  introducing the pressure $p=\lambda \nabla \cdot \boldsymbol{u}$ as an independent unknown. As a result, the linear elasticity equations \eqref{eq:elasticity-one} and \eqref{eq:elasticity-bd}  can be  rewritten as a mixed form:
\begin{eqnarray}
 - \epsilon \Delta\boldsymbol{u}  -\frac{\lambda+ \epsilon}{\lambda} \nabla p  &=&\boldsymbol{f}, \quad \text{in}\,\, \Omega \label{eq:uvp-laplace}\\
   \frac{\lambda+ \epsilon}{\lambda} \nabla \cdot \boldsymbol{u} -\frac{\lambda+ \epsilon}{\lambda^2}p&=&0, \quad \text{in}\,\, \Omega \label{eq:uvp-grad}\\
  \boldsymbol{u}&=&0,\quad \text{on}\,\, \partial \Omega \label{eq:elasticity-uv-bd}
\end{eqnarray}
Note that  \eqref{eq:uvp-grad} has been scaled by $\frac{\lambda+ \epsilon}{\lambda^2}$  to make the system symmetric. 

In the literature, finite element \cite{da2013virtual,nakshatrala2008finite,zhang1997analysis}  and finite difference \cite{da2010mimetic} methods have been applied to the linear elasticity problems. For example, Lee et al. \cite{lee2003locking} presented  a locking-free nonconforming finite element method for \eqref{eq:elasticity-one} and \eqref{eq:elasticity-bd}. Recently, a marker and cell scheme (MAC) for the mixed forms \eqref{eq:uvp-laplace}, \eqref{eq:uvp-grad}, and \eqref{eq:elasticity-uv-bd}, was proposed with stability and convergence analysis \cite{rui2018locking}.  Designing fast robust  solvers for the resulting discrete linear system is often challenging due to the  Lam\'{e} constants. An algebraic multigrid method for  linear elasticity problems was presented by Griebel et al. \cite{griebel2003algebraic}.  There are two main types of solvers for the discrete systems: preconditioning and multigrid. Preconditioning can be found in, for example,  Schwarz preconditioners  \cite{da2013isogeometric}, element-based preconditioner \cite{augarde2006element}.    Multigrid methods for discontinuous Galerkin discretizations of the  linear elasticity equations was presented by  Hong et al. \cite{hong2016robust},  where block Gauss–Seidel smoother is used.  Cai et al. \cite{cai1998first} proposed two first-order system least-squares approaches for the solution of the pure traction problem in planar linear elasticity. There are lots of  studies on numerical solvers \cite{baker2010improving,klawonn2004preconditioner,xiao2011robust}.

The work by Rui and Sun \cite{rui2018locking} focuses on the convergence order analysis of the MAC scheme for the linear elasticity problems. To the best of our knowledge, there seems no study of multigrid method for such discretization.  The contribution of this work is that we propose robust multigrid  methods with MAC scheme for  linear elasticity problems.  The proposed three block-structured  Braess Sarazj  relaxation  (BSR) schemes are Jacobi-BSR, Mass-BSR, and Vanka-BSR.  We apply local Fourier analysis (LFA) to help properly choose algorithmic parameters and derive optimal LFA smoothing factors for the proposed BSR-type relaxation schemes. Our findings show that the optimal smoothing factors are independent of  Lam\'{e} constants and meshsize. Our W-cycle and V-cycle multigrid methods are robust to  Lam\'{e} constants and meshsize, and are highly efficient. The quantitative convergence factors indicate that Vanka-BSR is the most efficient among three.

The rest of the work is organized as follows. In Section \ref{sec:discretization}, we introduce MAC scheme for the linear elasticity problem and   propose three block-structured Braess-Sarazin relaxation schemes. In Section \ref{sec:LFA}, we apply local Fourier analysis to study the smoothing, and  derive optimal smoothing factor for  the three block-structured schemes.  In Section \ref{sec:numerical-result}, we report numerical results to confirm our theoretical analysis. Finally, we draw conclusions in Section \ref{sec:conclusion}.
 
\section{Discretization and relaxation} \label{sec:discretization}
We now revisit the MAC scheme for the mixed   elasticity equations, \eqref{eq:uvp-laplace}, \eqref{eq:uvp-grad}, and \eqref{eq:elasticity-uv-bd}.  We carry out the discretization on  the uniform mesh of the unit domain $\Omega=[0,1]\times[0,1]$ with $h_{x}=h_{y}=h=1/N$,  and define the mesh points for the discrete unknowns as
\begin{equation*}
	(x_r, y_{\ell}) =( r h,  \ell h), \quad \text{where}\,\, r, \ell = 0, 1/2, 1,\dots, N-1, N-1/2, N.
\end{equation*}
Using  MAC scheme  \cite{rui2018locking},
the discrete unknowns, $u, v, p$,  in  \eqref{eq:uvp-laplace} and \eqref{eq:uvp-grad}, are placed at different positions on the grid:  the velocity $u$ is defined at the middle points of vertical edges ($\Box$), 
the velocity $v$ is defined at the middle points of horizontal edges ($\circ$), 
and the pressure $p$ is placed in the center of each cell ($\smallstar$), 
see Figure~\ref{fig:MAC-elasticity}. 

 \begin{figure}[!htp]
\centering
\includegraphics[width=0.5\textwidth]{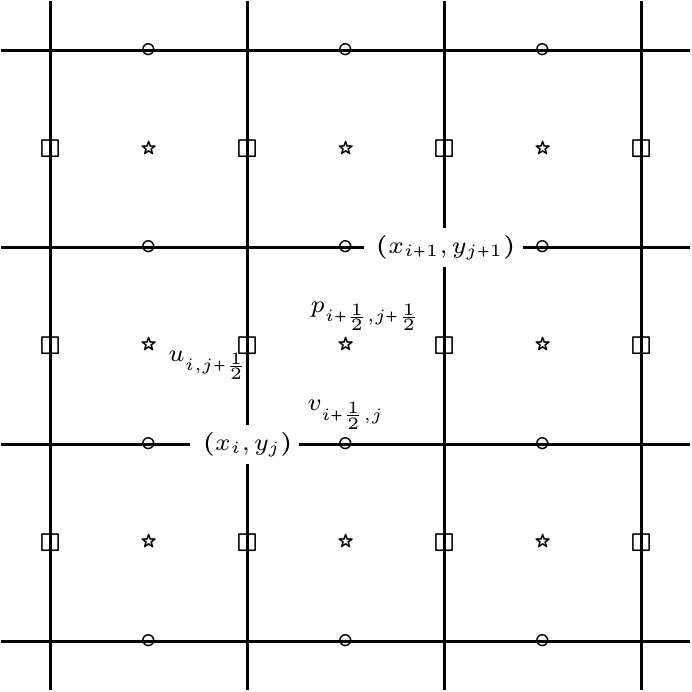}
 \caption{The staggered location of unknowns: $\Box-u,\,\, \circ-v,  \,\, \smallstar-p$.}\label{fig:MAC-elasticity}
\end{figure}

For simplicity, let $\tau =\frac{\lambda+ \epsilon}{\lambda}$. Then, the stencil representation of the MAC scheme for  \eqref{eq:uvp-laplace} and \eqref{eq:uvp-grad} is
\begin{equation}\label{eq:Lh-operator}
\mathcal{L}_h = \begin{pmatrix}
- \epsilon  \Delta_{h} & 0 & -\tau (\partial_x)_{h/2} \\
0 &  - \epsilon  \Delta_{h}  & -\tau (\partial_y)_{h/2} \\
\tau (\partial_x)_{h/2}  & \tau (\partial_y)_{h/2}  & -\tau/\lambda
\end{pmatrix},
\end{equation} 
where
\begin{equation}\label{eq:stencil-operator}
   -\Delta_{h} =\frac{1}{h^2}\begin{bmatrix}
      0 & -1 & 0\\
      -1 & 4 & -1 \\
      0 & -1 & 0
    \end{bmatrix},\quad
     (\partial_{x})_{h/2} =\frac{1}{h}\begin{bmatrix}
       -1 & 1 \\
    \end{bmatrix},\quad
     (\partial_{y})_{h/2} =\frac{1}{h}\begin{bmatrix}
       1 \\
      -1
    \end{bmatrix}.
\end{equation}
We now detail the MAC scheme to obtain the resulting linear system.  We set unknown matrices to be
\begin{equation*}
	U = \begin{pmatrix}
		u_{0,\frac{1}{2}} & u_{0,\frac{3}{2}} & \cdots & u_{0,\frac{2N-1}{2}}\\
		u_{1,\frac{1}{2}} & u_{1,\frac{3}{2}} & \cdots & u_{1,\frac{2N-1}{2}}\\
		\vdots & \vdots &        & \vdots \\
		u_{N,\frac{1}{2}} & u_{N,\frac{3}{2}} & \cdots & u_{N,\frac{2N-1}{2}}
	\end{pmatrix},\
	V = \begin{pmatrix}
		v_{\frac{1}{2},0} & v_{\frac{1}{2},1} & \cdots & v_{\frac{1}{2},N}\\
		v_{\frac{3}{2},0} & v_{\frac{3}{2},1} & \cdots & v_{\frac{3}{2},N}\\
		\vdots & \vdots &        & \vdots \\
		v_{\frac{2N-1}{2},0} & v_{\frac{2N-1}{2},1} & \cdots & v_{\frac{2N-1}{2},N}
	\end{pmatrix},
\end{equation*}
and
\begin{equation*}
	P = \begin{pmatrix}
		p_{\frac{1}{2},\frac{1}{2}} & p_{\frac{1}{2},\frac{3}{2}} & \cdots & p_{\frac{1}{2},\frac{2N-1}{2}}\\
		p_{\frac{3}{2},\frac{1}{2}} & p_{\frac{3}{2},\frac{3}{2}} & \cdots & p_{\frac{3}{2},\frac{2N-1}{2}}\\
		\vdots & \vdots &        & \vdots \\
		p_{\frac{2N-1}{2},\frac{1}{2}} & p_{\frac{2N-1}{2},\frac{3}{2}} & \cdots & p_{\frac{2N-1}{2},\frac{2N-1}{2}}\\
	\end{pmatrix},
\end{equation*}
where
$u_{i,j+\frac{1}{2}}$,
$v_{i+\frac{1}{2},j}$ and 
$p_{i+\frac{1}{2},j+{\frac{1}{2}}}$ with $i, j \in \mathbb{Z}$,
are used to approximate
$u(x_{i},y_{j+\frac{1}{2}})$,
$v(x_{i+\frac{1}{2}},y_{j})$ and
$p(x_{i+\frac{1}{2}},y_{j+\frac{1}{2}})$,
respectively.
Let  $S=(s_{ij}), T=(t_{ij})\in\mathbb{R}^{m\times n}$. We denote the sum over element wise multiplication of $S$ and $T$ as follows
\begin{equation*}
	S\odot T := \sum_{i=1}^m\sum_{j=1}^n s_{ij}t_{ij}.
\end{equation*}
For simplicity, let us treat the stencil operators in \eqref{eq:stencil-operator} as  matrices.  Then, the   discrete linear equations of \eqref{eq:uvp-laplace} and \eqref{eq:uvp-grad} are
\begin{eqnarray}
-\epsilon\Delta_{h} \odot
\begin{pmatrix}
u_{i-1,j+\frac{3}{2}} & u_{i,j+\frac{3}{2}} & u_{i+1,j+\frac{3}{2}} \\
u_{i-1,j+\frac{1}{2}} & u_{i,j+\frac{1}{2}} & u_{i+1,j+\frac{1}{2}} \\
u_{i-1,j-\frac{1}{2}} & u_{i,j-\frac{1}{2}} & u_{i+1,j-\frac{1}{2}}
\end{pmatrix}
-\tau (\partial_{x})_{h/2}\odot
\begin{pmatrix}
p_{i-\frac{1}{2},j+\frac{1}{2}} & p_{i+\frac{1}{2},j+\frac{1}{2}}
\end{pmatrix}
&=&f_{i,j+\frac{1}{2}}^1\label{eq:dis-moment1}\\
-\epsilon\Delta_{h} \odot
\begin{pmatrix}
v_{i-\frac{1}{2},j+1} & v_{i+\frac{1}{2},j} & v_{i+\frac{3}{2},j+1} \\
v_{i-\frac{1}{2},j} & v_{i+\frac{1}{2},j} & v_{i+\frac{3}{2},j} \\
v_{i-\frac{1}{2},j-1} & v_{i+\frac{1}{2},j} & v_{i+\frac{3}{2},j-1} \\
\end{pmatrix}
-\tau (\partial_{y})_{h/2}\odot
\begin{pmatrix}
	p_{i+\frac{1}{2},j+\frac{1}{2}} \\ p_{i+\frac{1}{2},j-\frac{1}{2}}
\end{pmatrix}
&=&f_{i+\frac{1}{2},j}^2 \label{eq:dis-moment2}\\
\tau (\partial_{x})_{h/2}\odot
\begin{pmatrix}
u_{i,j+\frac{1}{2}} & u_{i+1,j+\frac{1}{2}}
\end{pmatrix}
+\tau (\partial_{y})_{h/2}\odot
\begin{pmatrix}
v_{i+\frac{1}{2},j+1} \\ v_{i+\frac{1}{2},j}
\end{pmatrix}
-\frac{\tau}{\lambda}p_{i+\frac{1}{2},j+\frac{1}{2}}
&=&0, \label{eq:dis-div}
\end{eqnarray}
where $f_{i,j+\frac{1}{2}}^1$ and $f_{i+\frac{1}{2},j}^2$ are approximations 
of $f^1(x_i,y_{j+\frac{1}{2}})$ and $f^2(x_{i+\frac{1}{2}},y_j)$, respectively.   We remark that  modifications of the discretizations of \eqref{eq:dis-moment1},   \eqref{eq:dis-moment2},  and  \eqref{eq:dis-div} near boundaries are needed \cite{chen2018programming}. We omit it here.

In addition, the homogeneous boundary condition  \eqref{eq:elasticity-uv-bd} can be expressed as
\begin{equation*}
	u_{0,\frac{1}{2}} = u_{0,\frac{3}{2}} = \cdots = u_{0,\frac{2N-1}{2}}=
	u_{N,\frac{1}{2}} = u_{N,\frac{3}{2}} = \cdots = u_{N,\frac{2N-1}{2}} = 0,
\end{equation*}
and
\begin{equation*}
	v_{\frac{1}{2},0} = v_{\frac{3}{2},0} = \cdots = v_{\frac{2N-1}{2},0}=
	v_{\frac{1}{2},N} = v_{\frac{3}{2},N} = \cdots = v_{\frac{2N-1}{2},N} = 0.
\end{equation*}

Eliminating the boundary unknowns of $U$ and $V$, we obtain the following linear system 
\begin{equation}\label{eq:saddle-linear-system}
 \begin{pmatrix}
     \mathcal{A} & B^{T}\\
     B & -C\\
 \end{pmatrix}
        \begin{pmatrix} \boldsymbol{u}_{h} \\ p_{h}\end{pmatrix}
				=\begin{pmatrix} \boldsymbol{f}_h \\ 0 \end{pmatrix}={b}_h,
 \end{equation}
where $\mathcal{A}$ is the discrete representation of the scaled vector  Laplacian operator  
$-\epsilon \Delta\boldsymbol{u}$,  the matrix block $B^T$ is the negative discrete gradient, and $B$ the  discrete divergence.  $C$ contains the term $\frac{\tau}{\lambda}$.  Notice that in practice operator $C$ is an almost zero block because $\lambda$ is very large. 
$\boldsymbol{u}_{h}=( u_h , v_h )^T$ with 
\begin{equation*}
	u_h = \mbox{vec}(U),\ v_h = \mbox{vec}(V),\ p_h = \mbox{vec}(P),
\end{equation*}
where
\begin{align*}
	 \mbox{vec}(U) &= (u_{1,1/2},\dots,u_{N-1,1/2}, \, u_{1, 3/2}, \dots, u_{N-1, 3/2}, \dots,u_{1, N-1/2}, \dots, u_{N-1, N-1/2})^{T} \in \mathbb{R}^{(N-1)N}, \\
	 \mbox{vec}(V) &= (v_{1/2,1},\dots, v_{N-1/2,1}, \,v_{1/2, 2}, \dots, v_{N-1/2,2}, \dots, v_{1/2, N-1}, \dots, v_{N-1/2, N-1})^{T} \in \mathbb{R}^{N(N-1)}, \\
	\mbox{vec}(P) &= (p_{1/2,1/2},\dots, p_{N-1/2,1/2}, \,p_{1/2,3/2}, \dots, p_{N-1/2,3/2}, \dots, p_{1/2,N-1/2}, \dots,p_{N-1/2, N-1/2})^{T} \in \mathbb{R}^{N^2}.
\end{align*}

 \subsection{Braess-Sarazin relaxation} 
In this subsection, we present three block-structured multigrid relaxation schemes for solving \eqref{eq:saddle-linear-system}.   Solving the Laplacian  term efficiently plays an important role in constructing fast  smoothers for  \eqref{eq:saddle-linear-system}.  We consider  Braess-Sarazin-type relaxation \cite{braess1997efficient} (BSR), that is,  BSR smoother has the following structure:

\begin{equation}\label{eq:Precondtion}
   \mathcal{M}=  \begin{pmatrix}
         \epsilon D & B^{T}\\
     B & -C\\
    \end{pmatrix},
\end{equation}
where $ \epsilon D$, whose inverse is easy to compute, is an approximation to $\mathcal{A}$. We propose three choices of $D$, named Jacobi version, mass version and element-wise Vanka version, which lead to three types of $\mathcal{M}$. We now introduce them.

{\bf  Jacobi-BSR:} A simple choice for $D$ is  Jacobi smoother,  that is, 

\begin{equation}\label{eq:Jacobi-D}
   D   =  D_J=
\begin{pmatrix}
 {\rm diag}(A) &  0\\
0 &    {\rm diag}(A) 
\end{pmatrix},
\end{equation}
where  $A$ is the discrete scalar Laplacian.  We refer to the corresponding $\mathcal{M}$ as $\mathcal{M}_J$ and the relaxation scheme as J-BSR.

{\bf Mass-BSR}:   It was shown that the mass matrix obtained from bilinear finite elements is a good approximation to the inverse of the five-point discrete Laplacian operator \cite{CH2021addVanka}.  Moreover,  the mass-based smoother has been extended to the Stokes equations \cite{YH2021massStokes}, which leads to satisfactory convergence factor.  The mass stencil for bilinear discretization  in 2D is given by
\begin{equation}\label{eq:mass-2D-stencil}
  Q    =  \frac{h^2}{36}\begin{bmatrix}
   1 &  4    & 1\\
   4 &  16   & 4\\
   1 &  4    & 1
   \end{bmatrix}. 
\end{equation}
Then, for the mass-based Braess-Sarazin smoother $\mathcal{M}_m$,  $D^{-1}$ in \eqref{eq:Precondtion} is given by  
\begin{equation}\label{eq:C-inverse-Q}
D^{-1}=D_m^{-1}  = 
\begin{pmatrix}
 Q &  0\\
0 &   Q
\end{pmatrix}.
\end{equation}
The advantage of $\mathcal{M}_m$ is that we  only need $Q$ when we update the approximation in \eqref{eq:solution-of-precondtion}. We refer to the corresponding relaxation scheme for \eqref{eq:saddle-linear-system} as Q-BSR.
 
{\bf Vanka-BSR}: An additive element-wise Vanka smoother has  been proposed  to solve the Laplace problem efficiently \cite{CH2021addVanka}.   To solve $Az=b$ by additive Vanka relaxation scheme, the idea is to partition the unknowns of $z$ into  $m$ subdomains, then solve each subproblem independently.  The  additive Vanka relaxation scheme is
\begin{equation}
z_{k+1} =z_{k}+\omega M_e(b-A z_k),
\end{equation} 
where 
\begin{equation}
M_e = \sum_{j=1}^{m} V^T_j  W_j A_j^{-1}V_j,
\end{equation} 
in which $V_j$ is a restriction operator mapping the global unknowns (vector) to the $j$-th subproblem, and $W_j$ is a damping (diagonal) matrix, which shows how the solution of $j$-th subproblem is glued to  the global solution.  For the additive element-wise Vanka \cite{CH2021addVanka},  $A_j$ is a $4\times 4$ matrix, given by
\begin{equation*}
A_j = \frac{1}{h^2}
\begin{pmatrix}
4  & -1 & -1 &0\\
-1  & 4 & 0  &-1\\
-1  & 0 & 4  &-1\\
0  & -1 & -1 & 4
\end{pmatrix},
\end{equation*}
and 
\begin{equation*}
W_j =\frac{1}{4}
\begin{pmatrix}
1  & 0 & 0 &0\\
0  & 1 & 0  &0\\
0  & 0 & 1 &0\\
0  & 0 & 0 & 1
\end{pmatrix}.
\end{equation*} 
Then,  $D^{-1}$ in  \eqref{eq:Precondtion} is  given by
\begin{equation}\label{eq:V-D}
  D^{-1}= D^{-1}_v  = 
\begin{pmatrix}
 M_e &  0\\
0 &    M_e
\end{pmatrix},
\end{equation}
where  the stencil representation of $M_e$ \cite{CH2021addVanka}  is given by
\begin{equation}\label{eq:stencil-e-2d}
M_e = \frac{h^2}{96}
\begin{bmatrix}
 1& 4 & 1\\
4 & 28 & 4 \\
1 & 4 &1
\end{bmatrix}.
\end{equation} 
Note that 
\begin{equation*}
M_e =  \frac{3}{8}Q+ \frac{h^2}{8}I,
\end{equation*} 
which can help form global matrix $M_e$ using the tensor product of mass matrix in 1D. We refer to the corresponding relaxation scheme for \eqref{eq:saddle-linear-system} as V-BSR.

Based on smoother $\mathcal{M}$ and  a given approximation $\boldsymbol{y}_k$, the BSR-type relaxation scheme for solving \eqref{eq:saddle-linear-system} is
\begin{equation}\label{eq:relaxation-scheme}
\boldsymbol{y}_{k+1} = \boldsymbol{y}_k+\omega  \mathcal{M}^{-1}(b_h- \mathcal{L}_h\boldsymbol{y}_k),
\end{equation} 
where $\omega$ is a damping parameter to be determined. Then, the relaxation error operator for \eqref{eq:relaxation-scheme} is 
\begin{equation}\label{eq:relax-error-operator}
\mathcal{S}_h = I-\omega \mathcal{M}^{-1} \mathcal{L}_h.
\end{equation}

 In \eqref{eq:relaxation-scheme}, we can update the approximation by two steps. Let $\delta \boldsymbol{ y} =\mathcal{M}^{-1}(b_h- \mathcal{L}_{h} \boldsymbol{ y}_k)$. Then, the update 
 $\delta \boldsymbol{y}=(\delta _{\boldsymbol{u}}, \delta_p)$  is given by  
\begin{eqnarray}
  (C+B( \epsilon D)^{-1}B^{T})\delta_p&=&B ( \epsilon  D)^{-1} r_{\boldsymbol{u}}- r_{p}, \label{eq:solution-of-precondtion}\\
  \delta_{\boldsymbol{u}}&=& (  \epsilon D)^{-1}  (r_{\boldsymbol{u}}-B^{T}\delta_p),\nonumber
\end{eqnarray}
where $(r_{\boldsymbol{u}},r_{p})=b-\mathcal{L}\boldsymbol{y}_k$. Finally, $\boldsymbol{y}_{k+1} = \boldsymbol{y}_k+\omega \delta \boldsymbol{y}$.

Solving \eqref{eq:solution-of-precondtion} exactly is   expensive. In practice, we consider an inexact solve. For example, we  can apply a few  sweeps of weighted Jacobi iteration to the Schur complement system \eqref{eq:solution-of-precondtion}.  We refer to the inexact solve for the three versions  as J-IBSR, Q-IBSR, V-IBSR.  We  focus on the smoothing analysis for exact versions, then numerically test inexact versions.

\section{Smoothing analysis}\label{sec:LFA}

It is known that  the choice of multigrid components, such as relaxation parameter and grid-transfer operators, plays an important role in designing fast multigrid methods.  In relaxation scheme \eqref{eq:relaxation-scheme}, we have three choices of $\mathcal{M}$. We are curious that which gives the best performance and whether these smoothers are robust with respect to physical parameters, $\lambda$ and $\epsilon$. To answer these questions,
we employ LFA \cite{trottenberg2000multigrid,wienands2004practical}  to study and predict  multigrid performance. In LFA,  the relaxation error operator corresponds to a matrix {\it symbol}. Then, we can examine the spectral radius of the symbol of the relaxation error operator and  guide the choice of $\omega$.  In this process, we  minimize the largest spectral radii of the symbol  of all high frequencies, see \eqref{eq:H-L-frequency}, over real damping parameter to obtain the optimal {\it LFA smoothing factor}. Here, we consider multigrid methods with standard coarsening in  two-dimensions. The low and high frequencies for LFA are given by
\begin{equation}\label{eq:H-L-frequency}
\boldsymbol{\theta}=(\theta_1,\theta_2) \in T^{\rm L} = \left(-\frac{\pi}{2}, \frac{\pi}{2}\right], \quad  \boldsymbol{\theta}=(\theta_1,\theta_2)  \in T^{\rm H} = \left(\frac{\pi}{2}, \frac{3 \pi}{2}\right] \backslash T^{\rm L}. 
\end{equation}
We now introduce the symbol of an operator. 
\begin{definition} \cite{trottenberg2000multigrid}
Let $L_h =[s_{\boldsymbol{\kappa}}]_{h}$  be a scalar stencil operator acting on grid ${G}_{h}$ as
\begin{equation*}
  L_{h}w_{h}(\boldsymbol{x})=\sum_{\boldsymbol{\kappa}\in{V}}s_{\boldsymbol{\kappa}}w_{h}(\boldsymbol{x}+\boldsymbol{\kappa}h),
\end{equation*}
where  $s_{\boldsymbol{\kappa}}\in \mathbb{R}$ is constant,   $w_{h}(\boldsymbol{x}) \in l^{2} ({G}_{h})$, and  ${V}$ is a finite index set. 
Then, the  symbol of $L_{h}$ is 
\begin{equation}\label{eq:symbol-form}
 \widetilde{L}_{h}(\boldsymbol{\theta})=\displaystyle\sum_{\boldsymbol{\kappa}\in{V}}s_{\boldsymbol{\kappa}}e^{i \boldsymbol{\theta}\cdot\boldsymbol{\kappa}},\,\, i^2=-1. 
\end{equation} 
\end{definition} 
One of the most important ingredients  in LFA for multigrid methods is  LFA smoothing factor, which can predict  smoothing behavior.  
For simplicity, we drop subscript $h$ in the rest of the paper, unless it is necessary. 

 \begin{definition}
 The LFA smoothing factor for the relaxation error operator $\mathcal{S}$, see  \eqref{eq:relax-error-operator}, is given by
 \begin{equation}\label{eq:smoothing-mu}
\mu_{\rm loc}(\mathcal{S}(\omega)) = \max_{\boldsymbol{\theta} \in T^{\rm H}} \{\rho(\widetilde{\mathcal{S}}(\boldsymbol{\theta},\omega ) )\},
\end{equation} 
where $ \rho(\widetilde{\mathcal{S}})$ denotes the spectral radius of symbol $\widetilde{\mathcal{S}}$. 
 \end{definition}

 \begin{definition}
 The LFA optimal  smoothing factor for the relaxation error operator $\mathcal{S}$,  is defined as 
 \begin{equation}\label{eq:optimal-mu}
\mu_{\rm opt} = \min_{\omega \in \mathbb{R}^{+}} \mu_{\rm loc}(\mathcal{S}(\omega)).
\end{equation} 
 \end{definition}
In a two-grid method, the two-grid error operator can be expressed as
\begin{equation*}
E = \mathcal{S}^{\gamma_2}(I- \mathcal{P} (\mathcal{L}_H)^{-1}\mathcal{R}\mathcal{L})\mathcal{S}^{\gamma_1},
\end{equation*} 
where $\mathcal{L}_H$ ($H=2h$) is obtained from the rediscretization on the coarse grid,  $\mathcal{R}$ is  the restriction operator, and $\mathcal{P}$ is the interpolation operator. The integral numbers, $\gamma_1$ and $\gamma_2$, stand for the number of pre- and post-smoothing steps, respectively. 
For a two-grid method, fine and coarse grids operators are involved. To represent two-grid LFA convergence factor, we define
\begin{equation*}
\boldsymbol{\theta}^{\boldsymbol{\beta}}=(\theta_1^{\beta_1}, \theta_2^{\beta_2})=\boldsymbol{\theta}+(\beta_1 \pi, \beta_2 \pi),
\quad \boldsymbol{\theta} =\boldsymbol{\theta}^{00} \in T^{\rm L},
\end{equation*} 
where $\boldsymbol{\beta}=(\beta_1,\beta_2) \in \{(0,0), (1,0), (0,1), (1,1)\}$. Then, the symbol of two-grid error operator is   given by
\begin{equation*}
\widetilde{\mathbf{E}}(\boldsymbol{\theta}) = \widetilde{\mathbf{S}}^{\gamma_2}(\boldsymbol{\theta}) \left (\mathbf{I}- \widetilde{\mathbf{P}}(\boldsymbol{\theta}) \widetilde{\mathcal{L}}_H(2\boldsymbol{\theta}) ^{-1} \widetilde{\mathbf{R}}(\boldsymbol{\theta}) \widetilde{\mathbf{L}}(\boldsymbol{\theta}) \right) \widetilde{\mathbf{S}}^{\gamma_1}(\boldsymbol{\theta}), 
\end{equation*} 
where 
\begin{align*}
\widetilde{\mathbf{S}}(\boldsymbol{\theta}) & = {\rm diag} \{\widetilde{\mathcal{S}}(\boldsymbol{\theta}^{00}),\, \widetilde{\mathcal{S}}(\boldsymbol{\theta}^{10}), \, \widetilde{\mathcal{S}}(\boldsymbol{\theta}^{01}), \, \widetilde{\mathcal{S}}(\boldsymbol{\theta}^{11})\},\\
\widetilde{\mathbf{R}}(\boldsymbol{\theta}) & =  ( \widetilde{\mathcal{R}}(\boldsymbol{\theta}^{00}),\,  \widetilde{\mathcal{R}}(\boldsymbol{\theta}^{10}),\, \widetilde{\mathcal{R}}(\boldsymbol{\theta}^{01}),\, \widetilde{\mathcal{R}}(\boldsymbol{\theta}^{11})),\\
\widetilde{\mathbf{P}}(\boldsymbol{\theta}) & =  ( \widetilde{\mathcal{P}}(\boldsymbol{\theta}^{00});\, \widetilde{\mathcal{P}}(\boldsymbol{\theta}^{10});\, \widetilde{\mathcal{P}}(\boldsymbol{\theta}^{01});\, \widetilde{\mathcal{P}}(\boldsymbol{\theta}^{11})),
\end{align*}
in which ${\rm diag}\{T_1, T_2, T_3, T_4\}$  stands for the block diagonal matrix with diagonal blocks, $T_1, T_2, T_3$ and $T_4$.
\begin{definition}
 The two-grid LFA convergence factor for  $E$  is defined as 
 \begin{equation}\label{eq:LFA-rho}
\rho_h(\gamma) = \max_{\boldsymbol{\theta} \in T^{\rm L}} \{\rho(\widetilde{\mathbf{E}}(\boldsymbol{\theta},\omega ) )\},
\end{equation} 
where $ \rho(\widetilde{\mathbf{E}})$ denotes the spectral radius of symbol $\widetilde{\mathbf{E}}$. 
 \end{definition}
 
In many applications, the two-grid LFA convergence factor is almost the same as LFA smoothing factor, and both offer a sharp prediction of  actual two-grid   performance. Our main goal is to analytically find the optimal LFA smoothing factor \eqref{eq:optimal-mu} and optimal parameter, $\omega$, for the three types of BSR schemes discussed in Section \ref{sec:discretization}, and use these optimal $\omega$ to compute the corresponding two-grid LFA convergence factor, see \eqref{eq:LFA-rho}. 

Before deriving the optimal smoothing factor for the three relaxation schemes for the linear elasticity, we review the mass approximation to the scalar Laplacian.  Using \eqref{eq:symbol-form}, the symbol of $Q$, see \eqref{eq:mass-2D-stencil},   is
\begin{equation}\label{eq:symbol-mass-Q}
  \widetilde{Q}(\theta_1,\theta_2)  =\frac{h^2}{9}(4+2\cos\theta_1+2\cos\theta_2+\cos\theta_1\cos\theta_2),
\end{equation}
and the symbol of the five-point finite difference stencil for the scalar Laplacian is 
\begin{equation}\label{eq:scalar-symbol-Laplace}
\widetilde{A}  =\frac{4-2\cos\theta_{1}-2\cos\theta_2}{h^2}=:\frac{ \chi}{h^2}.
\end{equation}
It is obvious that $\widetilde{{\rm diag}( A )} = \frac{4}{h^2}=:\frac{ \eta_J}{h^2}$.  For simplicity, let  $\eta_m=(\widetilde{Q}/h^2)^{-1}$. Then  $\widetilde{Q}\widetilde{A}=\frac{  \chi}{\eta_m}$. 

Using  \eqref{eq:symbol-form}, the symbol of $M_e$ is
\begin{equation}\label{eq:symbol-Me}
  \widetilde{M}_e (\theta_1,\theta_2)  =\frac{h^2}{24}(7+2(\cos\theta_1+ \cos\theta_2)+\cos\theta_1\cos\theta_2). 
\end{equation}
For simplicity, let  $\eta_v=(\widetilde{M}_e/h^2)^{-1}$. Then  $\widetilde{M}_e\widetilde{A}=\frac{  \chi}{\eta_v}$.  We now review the optimal smoothing factor results   for the scalar Laplacian  \cite{CH2021addVanka}. 

\begin{lemma}\label{lemma:Jacobi-Laplace-opt-mu} \cite{he2018local}
Consider the relaxation error operator $S_J= I-\omega {\rm diag}(A) A $. For $\boldsymbol{\theta} \in T^{\rm H}$, we have
\begin{equation}\label{eq:JA-min-max-values}
 \widetilde{{\rm diag}( A )}\widetilde{A}=\frac{ \chi}{\eta_J}  \in \left[\frac{1}{2}, 2 \right].
\end{equation} 
Moreover, the optimal smoothing factor for $S_J$   is 
\begin{equation*}
\mu_{\rm opt}(S_J) = \min_{\omega} \max_{\boldsymbol{\theta} \in T^{\rm H}} \{|1-\omega  \widetilde{{\rm diag}( A )} \widetilde{A}_s|\} =\frac{3}{5}=0.600,
\end{equation*} 
provided that $\omega=\frac{4}{5}$.
\end{lemma}

\begin{lemma}\label{lemma:mass-Laplace-opt-mu} \cite{CH2021addVanka}
Consider the relaxation error operator $S_m= I-\omega  QA$. For $\boldsymbol{\theta} \in T^{\rm H}$, we have
\begin{equation}\label{eq:QA-min-max-values}
\widetilde{Q}\widetilde{A} = \frac{ \chi}{\eta_m}  \in \left[\frac{8}{9}, \frac{16}{9}\right].
\end{equation} 
Moreover, the optimal  smoothing factor for $S_m$   is 
\begin{equation*}
\mu_{\rm opt}(S_m) = \min_{\omega} \max_{\boldsymbol{\theta} \in T^{\rm H}} \{|1-\omega  \widetilde{Q} \widetilde{A}|\} =\frac{1}{3} \approx 0.333,
\end{equation*} 
provided that $\omega=\frac{3}{4}$.
\end{lemma}

\begin{lemma}\label{lemma:Vanka-Laplace-opt-mu} \cite{CH2021addVanka}
Consider the relaxation error operator $S_v= I-\omega M_e A $. For $\boldsymbol{\theta} \in T^{\rm H}$, we have
\begin{equation}\label{eq:Vanka-min-max-values}
 \widetilde{M}_e \widetilde{A}= \frac{ \chi}{\eta_v}  \in \left[\frac{3}{4}, \frac{4}{3}\right].
\end{equation} 
Moreover, the optimal  smoothing factor for $S_v$   is 
\begin{equation*}
\mu_{\rm opt}(S_v) = \min_{\omega} \max_{\boldsymbol{\theta} \in T^{\rm H}} \{|1-\omega  \widetilde{M}_e \widetilde{A}|\} =\frac{7}{25}= 0.280,
\end{equation*} 
provided that $\omega=\frac{24}{25}$.
\end{lemma}

Lemmas \ref{lemma:Jacobi-Laplace-opt-mu},  \ref{lemma:Vanka-Laplace-opt-mu}, and  \ref{lemma:mass-Laplace-opt-mu}   play an important role in our smoothing analysis for BSR for  the linear elasticity, since the Laplacian term appears in BSR.  It can be shown that the symbol of the discrete linear elasticity $\mathcal{L}$, see \eqref{eq:Lh-operator},  is 
\begin{equation*} 
   \widetilde{\mathcal{L}}(\theta_1,\theta_2) =\frac{1}{h^2}\begin{pmatrix}
        \epsilon  \chi       & 0                    & -i \tau 2h \sin\frac{\theta_1}{2}  \\
      0                   &   \epsilon   \chi          & -i \tau 2h \sin\frac{\theta_2}{2} \\
      i \tau 2h \sin\frac{\theta_1}{2}   &  i \tau 2h \sin\frac{\theta_2}{2}          & -\frac{\tau}{\lambda}h^2
    \end{pmatrix}.
\end{equation*}

For simplicity, denote $\eta$ be the symbol of $D$, where  $\eta=\eta_J, \eta_v, \eta_m$. Then,  the symbol of $\mathcal{M}$ is  
\begin{equation*} 
   \widetilde{\mathcal{M}}(\theta_1,\theta_2) =\frac{1}{h^2}\begin{pmatrix}
        \epsilon   \eta          & 0                      & -i \tau  2h \sin\frac{\theta_1}{2}  \\
      0                     &  \epsilon   \eta            &- i \tau 2h \sin\frac{\theta_2}{2} \\
      i \tau 2h \sin\frac{\theta_1}{2}   &  i  \tau 2h \sin\frac{\theta_2}{2}     & -\frac{\tau}{\lambda}h^2
    \end{pmatrix}.
\end{equation*}
Let $a=\tau  2h \sin\frac{\theta_1}{2},  b=\tau  2h \sin\frac{\theta_2}{2} $ and $c=\frac{\tau}{\lambda}h^2$.  To compute  the eigenvalues  $\sigma$ of $ \widetilde{\mathcal{M}}^{-1}  \widetilde{\mathcal{L}}$, we first compute the determinant of $ \widetilde{\mathcal{L}}-\sigma  \widetilde{\mathcal{M}}$:
 \begin{eqnarray*}
 | \widetilde{\mathcal{L}}-\sigma  \widetilde{\mathcal{M}}| &= &  
 \frac{1}{h^2}\begin{vmatrix}
        \epsilon \chi - \sigma \epsilon \eta         & 0                      & -a+\sigma a \\
      0                     &  \epsilon \chi  - \sigma \epsilon \chi          &-b+\sigma b \\
      a - \sigma a  &   b-\sigma b & -c+\sigma c
    \end{vmatrix}  \\
    &=& \frac{1}{h^2} ( \epsilon \chi  - \sigma \epsilon \eta  ) \left( ( \epsilon \chi  - \sigma \epsilon  \eta   ) (-c+\sigma c)+ (1-\sigma)^2(a^2+b^2)\right) \\
    &=& \frac{1}{h^2} (\epsilon \chi - \sigma \epsilon \eta  )(\sigma -1) \left( ( \epsilon \chi  - \sigma  \epsilon\eta  ) c- \chi(\sigma-1) \tau^2h^2\right) \\
    &=&  (\epsilon \chi  - \sigma  \epsilon \eta     )(\sigma -1) \left( ( \epsilon \chi - \sigma  \epsilon \eta  ) \frac{\tau}{\lambda}- \chi (\sigma-1) \tau^2\right).
 \end{eqnarray*}
 From the determinant, it can be shown  that the three eigenvalues of  $\widetilde{\mathcal{M}}^{-1}  \widetilde{\mathcal{L}}$ are 
 \begin{equation}\label{eq:three-eigs}
 1, \quad  \frac{ \chi}{ \eta}, \quad  \frac{ \tau \chi +  \epsilon \chi  /\lambda}{  \tau  \chi+  \epsilon \eta  /\lambda}=:\sigma^*.
 \end{equation}
From Lemmas \ref{lemma:Jacobi-Laplace-opt-mu}, \ref{lemma:mass-Laplace-opt-mu}, and \ref{lemma:Vanka-Laplace-opt-mu}, we know that the optimal smoothing factor corresponding to eigenvalues $1, \frac{ \chi}{ \eta}$.  
To explore the optimal smoothing factor for these three eigenvalues,  we first explore the range of $\sigma^*$ in \eqref{eq:three-eigs}.   Using $\tau =\frac{\lambda+ \epsilon }{\lambda}=1+\frac{\epsilon}{\lambda}$, we can simplify $\sigma^*$ as follows:
 
 \begin{eqnarray*}
\sigma^* &=& \frac{\tau  \chi +  \epsilon \chi /\lambda}{ \tau \chi  +  \eta   \epsilon /\lambda} \\
& =& \frac{ \tau +  \epsilon /\lambda}{ \tau +  \frac{\eta }{\chi}  \epsilon /\lambda} \\
& =& \frac{  1+\frac{ \epsilon }{\lambda} +  \frac{ \epsilon }{\lambda} }{ 1 + \frac{ \epsilon }{\lambda}  + \frac{\eta   \epsilon }{ \chi \lambda}  } \\
&=&  \frac{  1+2\frac{\epsilon}{\lambda}  }{ 1 + (1+ \frac{\eta }{ \chi})\frac{ \epsilon }{\lambda} }=\sigma^*\left(\frac{\eta}{\chi}\right).  
\end{eqnarray*} 
Since $\sigma^*\left(\frac{\eta}{\chi}\right)$ is a function of $\frac{\eta}{\chi}$, and we know the range of  $\frac{\chi}{\eta}$ from Lemmas \ref{lemma:Jacobi-Laplace-opt-mu},  \ref{lemma:mass-Laplace-opt-mu},  and \ref{lemma:Vanka-Laplace-opt-mu},  we now present the range of $\sigma^*\left(\frac{\eta}{\chi}\right)$ for $\eta =\eta_J, \eta_m, \eta_v$.
\begin{theorem}\label{thm:bound-sigma*}
 Let $0<t_1<1<t_2$.  Assume that  $\frac{\eta}{ \chi} \in [t_1, t_2]$, that is, $\frac{\chi}{ \eta} \in [1/t_2, 1/t_1]$. For  $\boldsymbol{\theta}\in T^{\rm H}$ , we have $\sigma^*\left(\frac{\eta}{\chi}\right) \in (1/t_2, 1/t_1)$.
Furthermore,

\begin{itemize}
\item $\sigma^*\left(\frac{\eta_J}{\chi}\right) \in (1/2, 2)$. 

\item $\sigma^*\left(\frac{\eta_m}{\chi}\right) \in (8/9, 16/9)$. 

\item $\sigma^*\left(\frac{\eta_v}{\chi}\right) \in (3/4, 4/3)$. 
\end{itemize}
\end{theorem}
 
 \begin{proof}

Note that $\sigma^*$ is a decreasing function  of  $\frac{\eta}{ \chi}$.   It follows that 
\begin{equation}\label{eq:t1-t2-bound}
  \frac{  1+2\frac{ \epsilon }{\lambda}  }{ 1 + (1+ t_2)\frac{ \epsilon }{\lambda} } =\sigma^*(t_2)  \leq   \sigma^*\left(\frac{\eta}{ \chi}\right)\leq  \sigma^*(t_1)=  \frac{  1+2\frac{ \epsilon }{\lambda}  }{ 1 + (1+ t_1)\frac{ \epsilon }{\lambda} }.
\end{equation} 
Let $t\in[t_1,t_2]$.  From
\begin{eqnarray*}
 \frac{  1+2\frac{ \epsilon }{\lambda}  }{ 1 + (1+ t)\frac{ \epsilon }{\lambda} }  -\frac{1}{t}=\frac{ (t-1)(1+ \epsilon /\lambda)}{t\left(1+(1+t)\epsilon /\lambda\right)},
\end{eqnarray*}
we have 
\begin{equation} 
 \frac{1}{t_2}< \frac{  1+2\frac{ \epsilon }{\lambda}  }{ 1 + (1+ t_2)\frac{ \epsilon }{\lambda} }   \leq   \sigma^* \leq    \frac{  1+2\frac{ \epsilon }{\lambda}  }{ 1 + (1+ t_1)\frac{ \epsilon }{\lambda} }<\frac{1}{t_1}.
\end{equation} 
It follows that $\sigma^*\in (1/t_2, 1/t_1)$.  Using Lemmas \ref{lemma:Jacobi-Laplace-opt-mu}, \ref{lemma:Vanka-Laplace-opt-mu}, \ref{lemma:mass-Laplace-opt-mu}, we obtain the desired results.
 \end{proof}
 
Theorem  \ref{thm:bound-sigma*} tells us that the range of $\sigma^*$ is a subset of the range of $\frac{\chi}{\eta}$.  From Lemmas \ref{lemma:Jacobi-Laplace-opt-mu}, \ref{lemma:Vanka-Laplace-opt-mu}, \ref{lemma:mass-Laplace-opt-mu}, we now present the optimal smoothing factors for the three BSR schemes.
 
 \begin{theorem} \label{thm:optimal-mu}
For the BSR schemes for the linear elasticity problems, we have the following smoothing results. 
\begin{itemize}
\item The optimal smoothing factor for  J-BSR   is
  \begin{equation*} 
   \mu_{{\rm opt} }=  \min_{\omega}\max_{ \boldsymbol{\theta}\in T^{{\rm H}}} \{\rho(\mathcal{\widetilde{S}} ( \omega,\boldsymbol{\theta}))\}=\frac{3 }{5}=0.600,
\end{equation*}
where the minimum is uniquely achieved at $\omega =\frac{4 }{5}$.

\item The optimal smoothing factor for  Q-BSR   is
  \begin{equation*} 
   \mu_{{\rm opt} }=  \min_{\omega}\max_{ \boldsymbol{\theta}\in T^{{\rm H}}} \{\rho(\mathcal{\widetilde{S}} ( \omega,\boldsymbol{\theta}))\}=\frac{1 }{3}\approx 0.333,
\end{equation*}
where the minimum is uniquely achieved at $\omega =\frac{3 }{4 }$.

\item The optimal smoothing factor for  V-BSR   is
  \begin{equation*} 
   \mu_{{\rm opt} }=  \min_{\omega}\max_{ \boldsymbol{\theta}\in T^{{\rm H}}} \{\rho(\mathcal{\widetilde{S}} ( \omega,\boldsymbol{\theta}))\}=\frac{7}{25}= 0.280,
\end{equation*}
where the minimum is uniquely achieved at $\omega =\frac{24 }{25}$.
\end{itemize}
\end{theorem}

 Theorem \ref{thm:optimal-mu} indicates that the optimal smoothing factors of the three BSR  schemes are  independent of physical parameters, $\lambda$ and $\epsilon$.

\section{Numerical experiment}\label{sec:numerical-result}
In this section, we first compute the two-grid LFA convergence factor defined in \eqref{eq:LFA-rho} using $\omega$ given in Theorem \ref{thm:optimal-mu}. Then, we present some numerical results for multigrid performance to validate the efficiency and robustness of the proposed algorithms. 
 
\subsection{LFA prediction}
In multigrid,  we consider restriction operators using  six points  for the $u$ and $v$ components of the velocity:
 \begin{equation*} 
   R_{h,u}   =    \frac{1}{8}\begin{bmatrix}
   1 &       & 1  \\
   2 &   \star   & 2   \\
  1 &      & 1  
   \end{bmatrix},
 \quad 
   R_{h,v}   =    \frac{1}{8}\begin{bmatrix}
   1 &    2  & 1  \\
    &   \star   &    \\
  1 &   2   & 1  
   \end{bmatrix},
\end{equation*}
where the $\star$ denotes the position (on the coarse grid) at which the discrete operator is applied. We take the interpolation to be  the scaled transpose of the restriction operators, that is,  $P_{h,u} =4 R^T_{h,u}$ and     $P_{h,v} =4 R^T_{h,v}$.
The simplest restriction operator  for the pressure is
 \begin{equation*} 
   R_{h,p}   =    \frac{1}{4}\begin{bmatrix}
   1 &      & 1  \\
    &   \star   &    \\
  1 &    & 1  
   \end{bmatrix},
\end{equation*}
and we consider $P_{h,p}=4R^T_{h,p}$.  Then, the restriction operator for the linear elasticity problems is
\begin{equation*}
\mathcal{R} =
\begin{pmatrix}
R_{h,u} & 0 &  0\\
0 & R_{h,v}& 0\\
 0 & 0 & R_{h,p}\\
\end{pmatrix},
\end{equation*} 
and $\mathcal{P} =4\mathcal{R}^T$. For other options of restriction and interpolation operators, we refer to the work of Niestegge and Witsch \cite{MR1049395}. 

We use $h=\frac{1}{128}$ to carry out  the two-grid LFA prediction  tests.   We test   two Lam\'{e} coefficients: 
\begin{equation*}
\epsilon =1  \quad \text{and}\quad  \lambda=\frac{2\nu}{1-2\nu},
\end{equation*} 
where $\nu = 0.45$ is in the compressible elasticity case and $\nu = 0.4999999$ is in the incompressible elasticity case. Tables \ref{tab:LFA-results-JBSR},
\ref{tab:LFA-results-QBSR}, and \ref{tab:LFA-results-VBSR} show the two-grid LFA convergence factor, $\rho_h(\gamma)$ (see \eqref{eq:LFA-rho}), as a function of the number of smoothing steps $\gamma$, for exact J-BSR, Q-BSR and V-BSR, respectively. We see that when $\gamma=1$, the two-grid LFA convergence factor is the same as the  optimal smoothing factor, as shown in Theorem \ref{thm:optimal-mu}, and $\rho_h(\gamma)\approx \mu^{\gamma}_{\rm opt}$, where $\mu_{\rm opt}$ is given  in Theorem \ref{thm:optimal-mu}.  All results show that the three methods are robust  to physical parameters $\epsilon$ and $\nu$.
 
 \begin{table}[H]
 \caption{Two-grid LFA convergence factors of exact J-BSR with different number of smoothing steps $\gamma$, and physical parameters,  $(\epsilon,\nu)$. }
\centering
\begin{tabular}{l cccc}
\hline
 $(\epsilon,\nu)$                                     &$\gamma=1$      &$\gamma=2$      &$\gamma=3$    & $\gamma=4$  \\ \hline
 
$(1,0.45)$                                             &0.600              &0.360              &0.216      &0.130    \\
$(1,0.4999999)$                                    &0.600              &0.360              &0.216      &0.130    \\
\hline
\end{tabular}\label{tab:LFA-results-JBSR}
\end{table}

 \begin{table}[H]
 \caption{Two-grid LFA convergence factors of exact  Q-BSR with different number of smoothing steps $\gamma$, and physical parameters,  $(\epsilon,\nu)$. }
\centering
\begin{tabular}{l cccc}
\hline
 $(\epsilon,\nu)$                                     &$\gamma=1$      &$\gamma=2$      &$\gamma=3$    & $\gamma=4$  \\ \hline
 
$(1,0.45)$                                             &0.333               &0.111               &0.037       &  0.028    \\
$(1,0.4999999)$                                   &0.333               &0.111               &0.037       &  0.028    \\
\hline
\end{tabular}\label{tab:LFA-results-QBSR}
\end{table}

 \begin{table}[H]
 \caption{Two-grid LFA convergence factors of exact  V-BSR with different number of smoothing steps $\gamma$, and physical parameters,  $(\epsilon,\nu)$. }
\centering
\begin{tabular}{l cccc}
\hline
 $(\epsilon,\nu)$                                     &$\gamma=1$      &$\gamma=2$      &$\gamma=3$    & $\gamma=4$  \\ \hline
 
$(1,0.45)$                                             &  0.280                &0.096               &0.056       &0.044     \\
$(1,0.999999)$                                     &  0.280                &0.096               &0.056       &0.044   \\
\hline
\end{tabular}\label{tab:LFA-results-VBSR}
\end{table}

\subsection{Multigrid performance}
We consider the model problems \eqref{eq:elasticity-one} and \eqref{eq:elasticity-bd} in a unit square domain \cite{rui2018locking}. The exact solution is given by
\begin{align*}
u(x,y) &= (-1+\cos(2\pi x)) \sin(2\pi y) +\frac{1}{1+\lambda} \sin(\pi x) \sin(\pi y),\\
v(x,y) &= (1-\cos(2\pi y)) \sin(2\pi x) +\frac{1}{1+\lambda} \sin(\pi x) \sin(\pi y).
\end{align*}
The source term $\boldsymbol{f}=(f^1,f^2)^{T}$ is computed via $\boldsymbol{f}=- \epsilon \Delta\boldsymbol{u} -(\lambda+ \epsilon) \nabla (\nabla\cdot\boldsymbol{u})$, and
\begin{align*}
f^1 &= \frac{(\lambda+3\epsilon) \pi^2}{1+\lambda}\sin(\pi x)\sin(\pi y) +8\pi^2\epsilon \sin(2\pi y) \cos(2\pi x)- 4\pi^2\epsilon \sin(2\pi y)-\frac{(\lambda+\epsilon)\pi^2}{1+\lambda} \cos(\pi x) \cos(\pi y),\\
f^2&= \frac{(\lambda+3\epsilon) \pi^2}{1+\lambda}\sin(\pi x)\sin(\pi y) -8\pi^2\epsilon \sin(2\pi y) \sin(2\pi x)+4\pi^2\epsilon \sin(2\pi x)-\frac{(\lambda+\epsilon)\pi^2}{1+\lambda} \cos(\pi x) \cos(\pi y).
\end{align*}
It is well known that for MAC scheme \cite{rui2018locking}, 
the displacement and the pressure have both second convergence order on uniform grids. To show the convergence order, we denote that
\begin{align*}
	&\|u-u_h\|_0 = \sqrt{ \sum_{i=1}^{N-1}\sum_{j=1}^N \Big(u(x_i,y_{j-\frac{1}{2}})-u_{i,j-\frac{1}{2}}\Big)^2 h^2}  \\
	&\|v-v_h\|_0 = \sqrt{ \sum_{i=1}^N\sum_{j=1}^{N-1} \Big(v(x_{i-\frac{1}{2}},y_i)-v_{i-\frac{1}{2},j}\Big)^2 h^2}  \\
	&\|p-p_h\|_0 = \sqrt{ \sum_{i=1}^N\sum_{j=1}^N \Big(p(x_{i-\frac{1}{2}},y_{j-\frac{1}{2}})-p_{i-\frac{1}{2},j-\frac{1}{2}}\Big)^2 h^2}.
\end{align*}
We present  the convergence results  with two pairs $(\epsilon, \nu)=(1,0.45)$ and $(\epsilon, \nu)=(1,0.4999999)$.  Figure \ref{fig: second-order-case1} shows that for these pairs, we obtain the expected second order convergence for both velocity and pressure.
 \begin{figure}[H]
\centering
\includegraphics[width=0.49\textwidth]{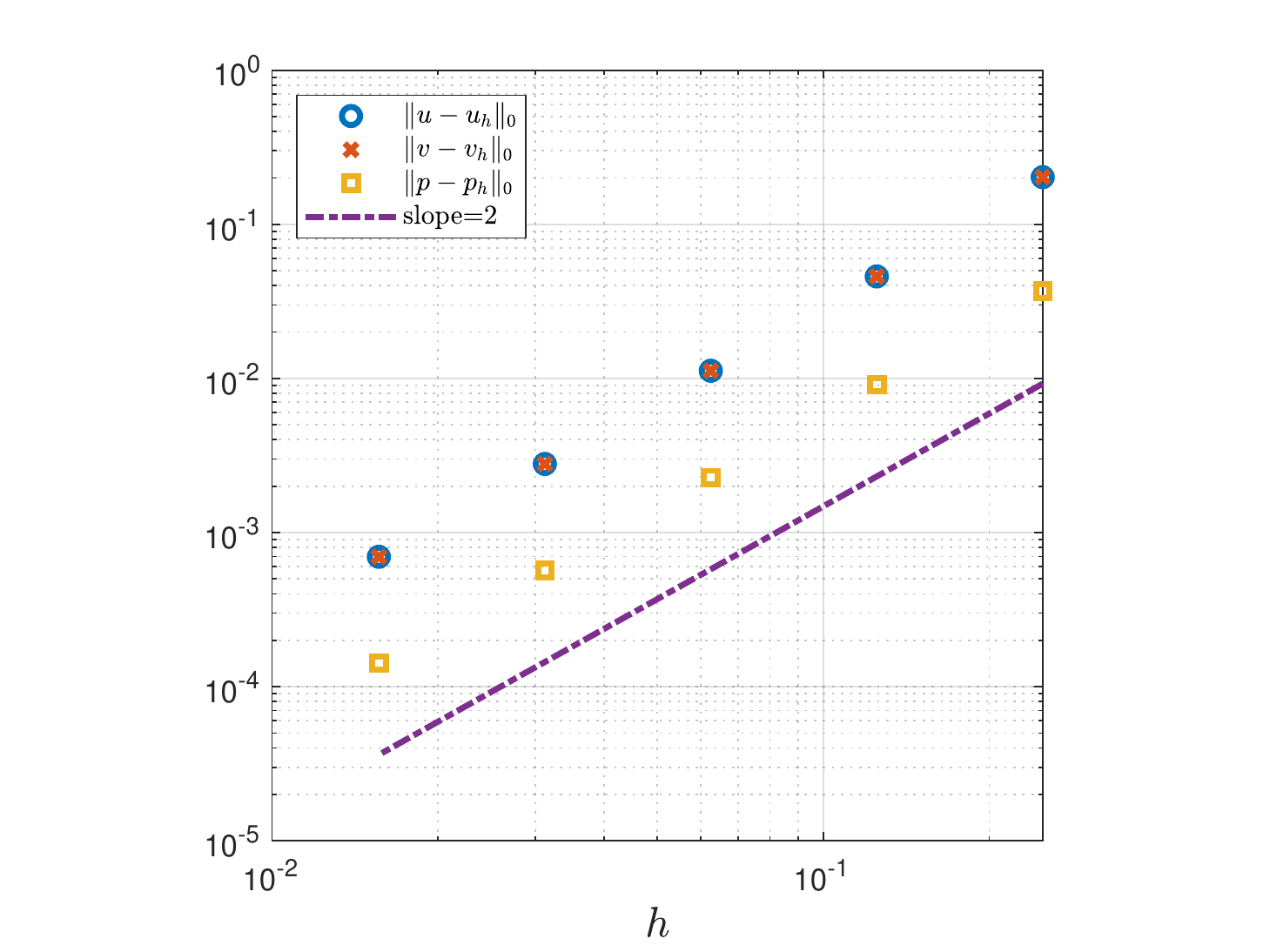}
\includegraphics[width=0.49\textwidth]{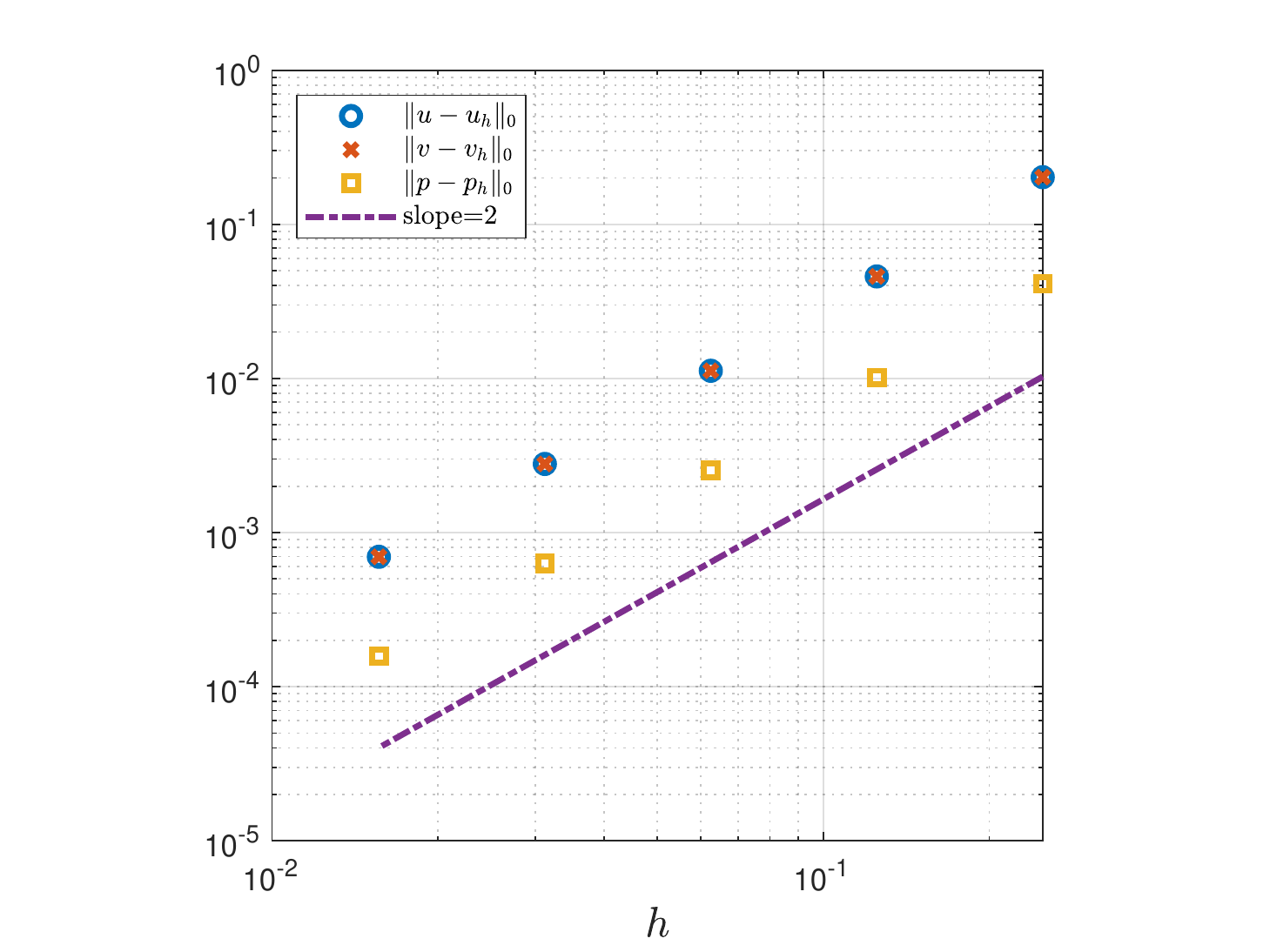}
 \caption{Convergence order. Left: $\epsilon =1$ and $\nu=0.45$. Right: $\epsilon =1$ and $\nu=0.4999999$.}\label{fig: second-order-case1}
\end{figure}
To validate our LFA predictions of the three block-structured multigrid methods, we measure the actual multigrid performance via 
\begin{equation*}
\hat{\rho}_h^{(k)} = \left(\frac{||r_k||}{||r_0||}\right)^{1/k},
\end{equation*} 
 where $r_k = b_h- \mathcal{L}_h \boldsymbol{\hat{y}}_k$  and $\boldsymbol{\hat{y}}_k$ is the $k$-th multigrid iteration.   In our test, we report  $\hat{\rho}_h^{(k)}$ with the smallest $k$ such that $||r_k||/||r_0||\leq 10^{-10}$. The initial guess is chosen randomly. The coarsest grid is a $4\times 4$ mesh.

As mentioned before, for the Schur complement system \eqref{eq:solution-of-precondtion},  we consider an inexact solve.   We apply weighted ($\omega_J$) Jacobi iteration to solve \eqref{eq:solution-of-precondtion}, where the iteration number of Jacobi relaxation is the one such that the 2-norm of the residual of the Schur complex system is less than 0.1 and it is at most three.  For J-IBSR, we use $\omega_J=0.8$.  For Q-IBSR, we use $\omega_J=0.8$.  For V-IBSR, we use $\omega_J=1$. The reason we choose these $\omega_J$ is that they give more robust multigrid methods.

\subsubsection{Two-grid  results}
We first report actual two-grid performance with finest meshsize $h=1/64$.  Table \ref{tab:TG-measured-convergence-small} presents the measured two-grid convergence factor as a function of the number of smoothing steps, $\gamma$, with $(\epsilon,\nu)=(1,0.45)$, for J-IBSR, Q-IBSR and V-IBSR. We see that the performance of 
J-IBSR and V-IBSR agree with the results  of two-grid LFA predictions of exact versions reported in Tables \ref{tab:LFA-results-JBSR} and  \ref{tab:LFA-results-VBSR}. However, we see the measured two-grid convergence factor  of Q-IBSR in Table \ref{tab:TG-measured-convergence-small} is a little larger than the two-grid LFA convergence factor of  exact Q-BSR shown in Table \ref{tab:LFA-results-QBSR}. This is not surprising, since we solve the Schur complement system inexactly.  We present the two-grid results of inexact versions for  $(\epsilon,\nu)=(1,0.4999999)$ in Table \ref{tab:TG-measured-convergence-large}, where we see  similar behavior as those in Table \ref{tab:TG-measured-convergence-small}.

\begin{table}[H]
 \caption{Measured two-grid convergence  factors of  J-IBSR, Q-IBSR and V-IBSR  with  $(\epsilon,\nu)=(1,0.45)$, and different number of smoothing steps $\gamma$. }
\centering
\begin{tabular}{l cccc}
\hline
method                                    &$\gamma=1$      &$\gamma=2$      &$\gamma=3$    & $\gamma=4$  \\ \hline
 
 J-IBSR                                   & 0.609               &0.376               &0.235        &0.152      \\
Q-IBSR                                     & 0.444            &0.232                 &0.121         &0.072     \\
V-IBSR                                   &0.267                 &0.083                &0.026       &0.019     \\
\hline
\end{tabular}\label{tab:TG-measured-convergence-small}
\end{table} 

 \begin{table}[H]
 \caption{Measured two-grid convergence  factors of  J-IBSR, Q-IBSR and V-IBSR  with  $(\epsilon,\nu)=(1,0.4999999)$, and different number of smoothing steps $\gamma$.}
\centering
\begin{tabular}{l cccc}
\hline
method                                    &$\gamma=1$      &$\gamma=2$      &$\gamma=3$    & $\gamma=4$  \\ \hline
 
 J-IBSR                                     &0.697                &0.377               &0.235       &0.151      \\
Q-IBSR                                     &0.451               &0.232              &0.119       &0.071     \\
V-IBSR                                     & 0.265            &0.082                &0.021       & 0.029    \\
\hline
\end{tabular}\label{tab:TG-measured-convergence-large}
\end{table}

\subsubsection{W(1,1)-cycle results}
The results of $\gamma=2$ in Table \ref{tab:TG-measured-convergence-large}  show  that  $0.376^{22}\approx 10^{-10}$ for J-IBSR,  $0.232^{15}\approx 10^{-10}$ for Q-IBSR, and  $0.096^{9}\approx 10^{-10}$ for V-IBSR, which means that it takes around $22, 15, 10$ iterations for J-IBSR, Q-IBSR, and V-IBSR, respectively,   to achieve the stopping criterion $||r_k||/||r_0||\leq 10^{-10}$.  

Next, we consider  W(1,1)-cycle and report the history of the convergence of  the three  BSR-type multigrid methods with $h=1/N$, where $N$ is the number of intervals in each dimension. We consider $N=32, 64, 128, 256$, and $(\epsilon,\nu)=(1,0.45)$ and  $(\epsilon,\nu)=(1,0.4999999)$.   Figure  \ref{fig: W-J-IBSR} shows the performance of J-IBSR W(1,1)-cycle with different finest meshsize  and physical parameters $(\epsilon, \nu)$. We see  robustness of J-IBSR W(1,1)-cycle multigrid, and it needs around 21 iterations to achieve the required stopping criterion, which is the same as those for  two-grid methods.  For  Q-IBSR W(1,1)-cycle, Figure \ref{fig: W-Q-IBSR} indicates that  16 or 17 iterations are enough to achieve the stopping criterion, which agrees with the two-grid results presented in Tables \ref{tab:TG-measured-convergence-small} and \ref{tab:TG-measured-convergence-large}. For V-IBSR W(1,1)-cycle, it needs only around 10 iterations to reach the stopping criterion, as shown in Figure \ref{fig: W-V-IBSR}. Comparing these three relaxation schemes, V-IBSR W(1,1)-cycle is the most efficient.

\begin{figure}[H]
\centering
\includegraphics[width=0.49\textwidth]{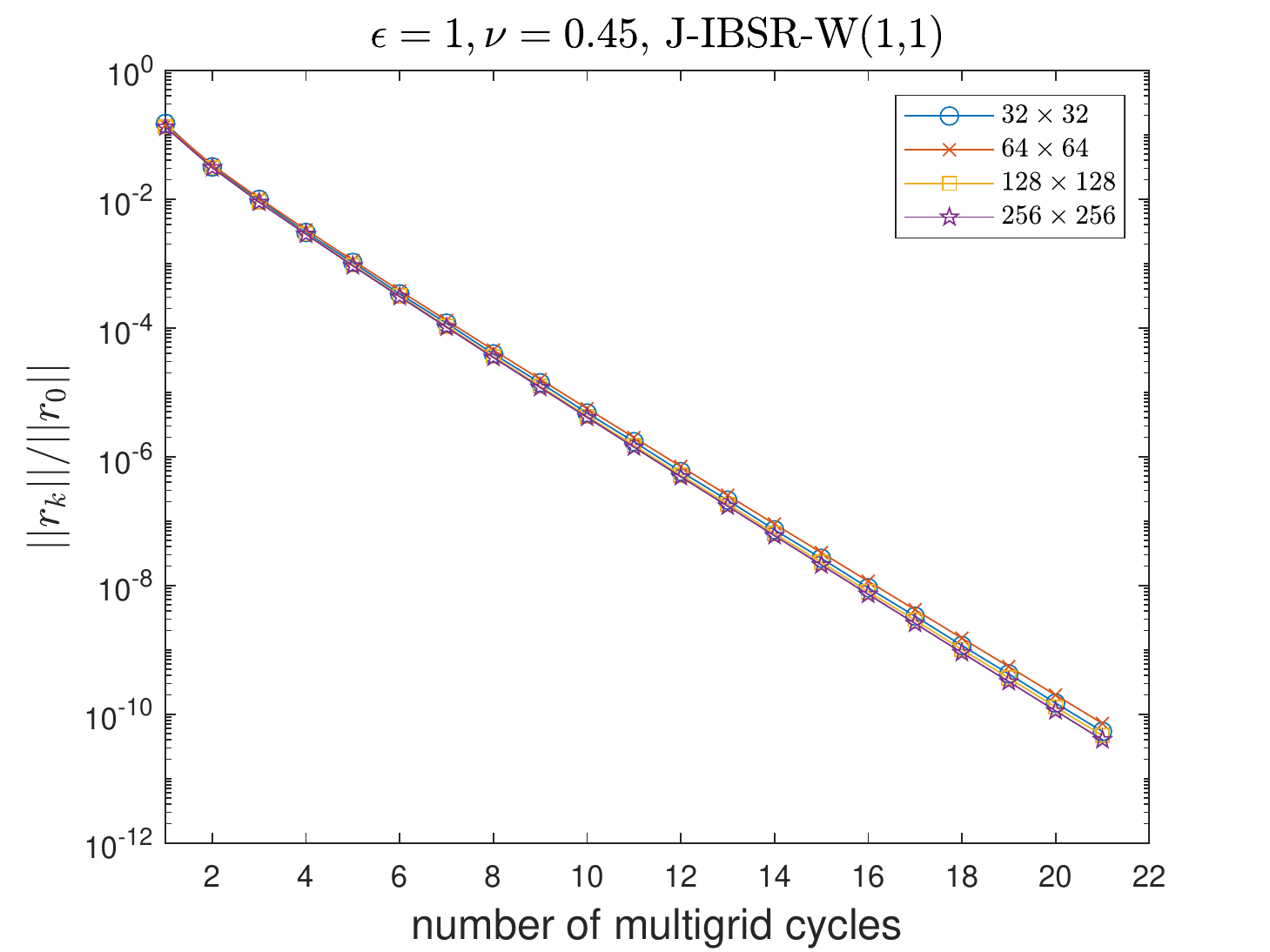}
\includegraphics[width=0.49\textwidth]{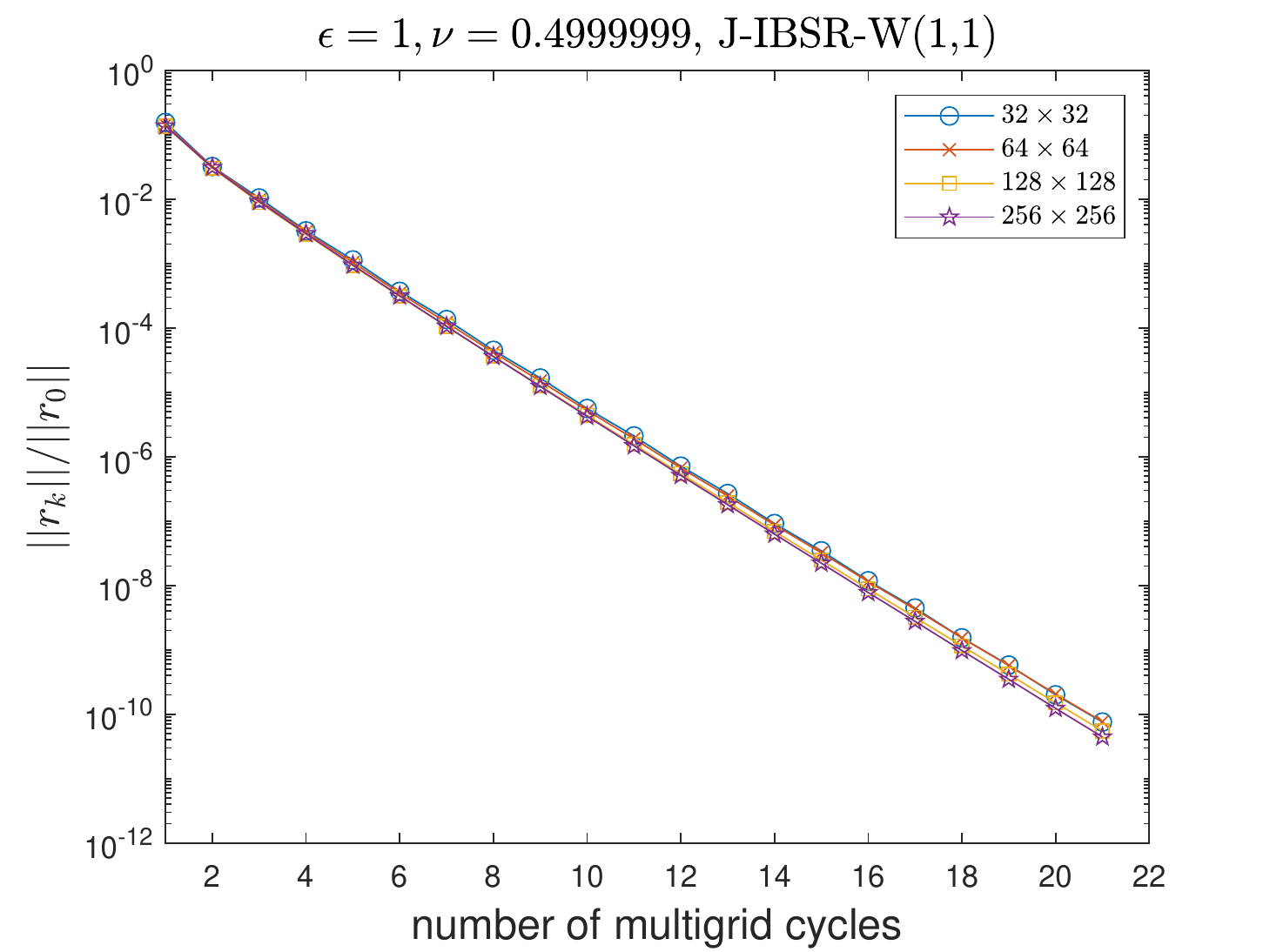}
 \caption{History of the convergence of  J-IBSR W(1,1)-multigrid method. Left: $\epsilon =1$ and $\nu=0.45$. Right: $\epsilon =1$ and $\nu=0.4999999$.}
 \label{fig: W-J-IBSR}
\end{figure}

\begin{figure}[H]
\centering
\includegraphics[width=0.49\textwidth]{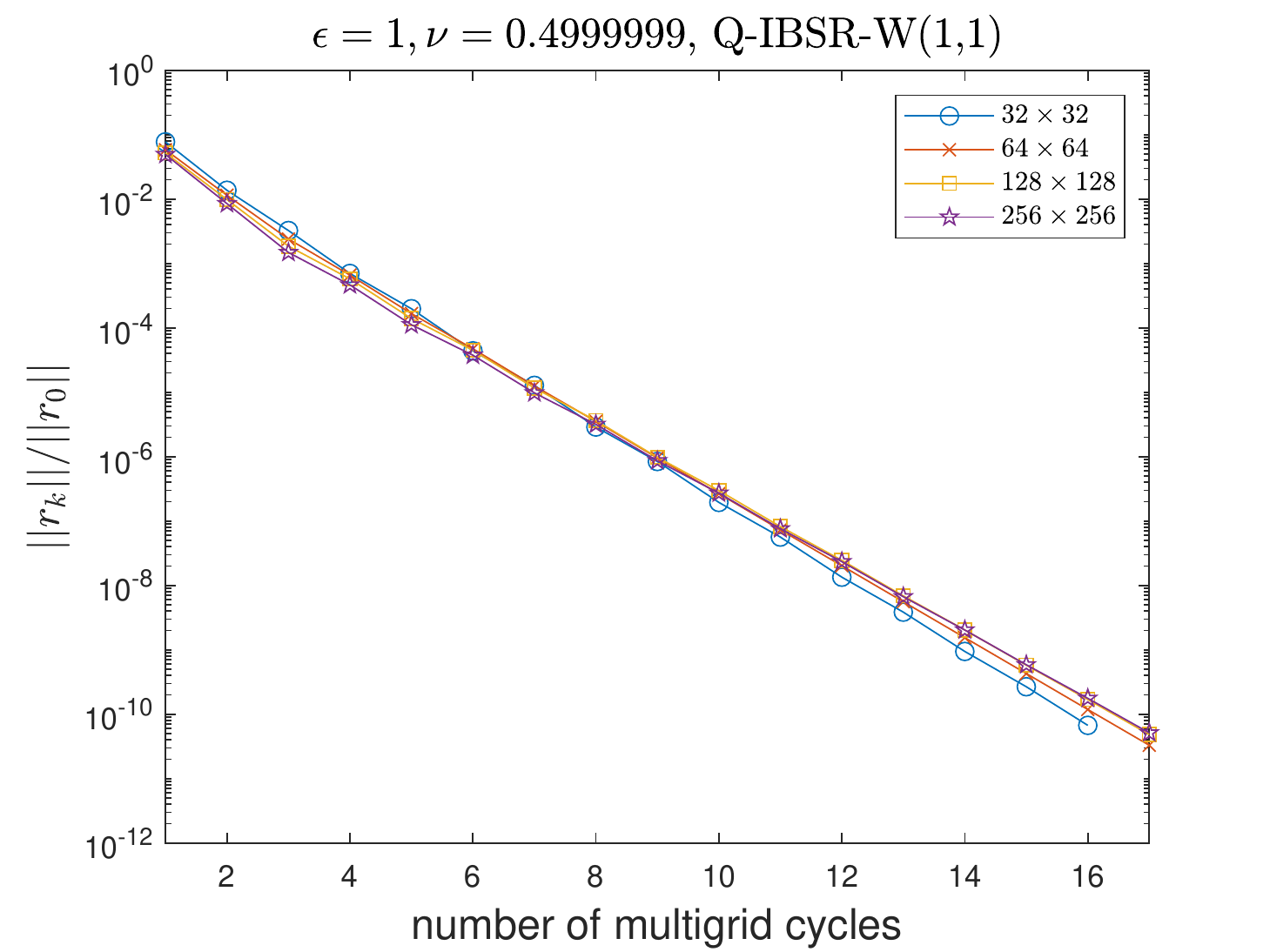}
\includegraphics[width=0.49\textwidth]{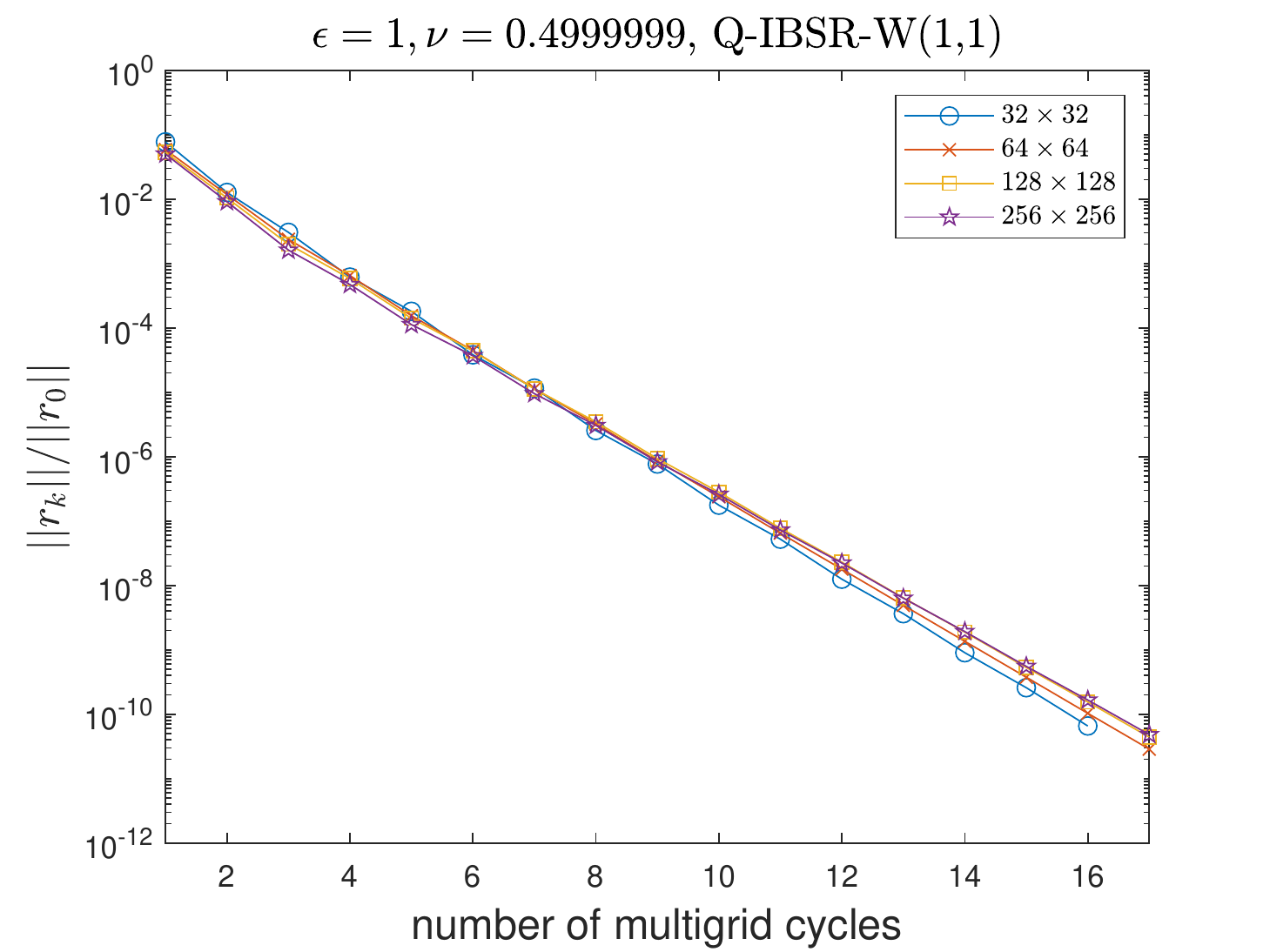}
 \caption{History of the convergence of  Q-IBSR W(1,1)-multigrid method.  Left: $\epsilon =1$ and $\nu=0.45$.  Right: $\epsilon =1$ and $\nu=0.4999999$.}\label{fig: W-Q-IBSR}
\end{figure}

\begin{figure}[H]
\centering
\includegraphics[width=0.49\textwidth]{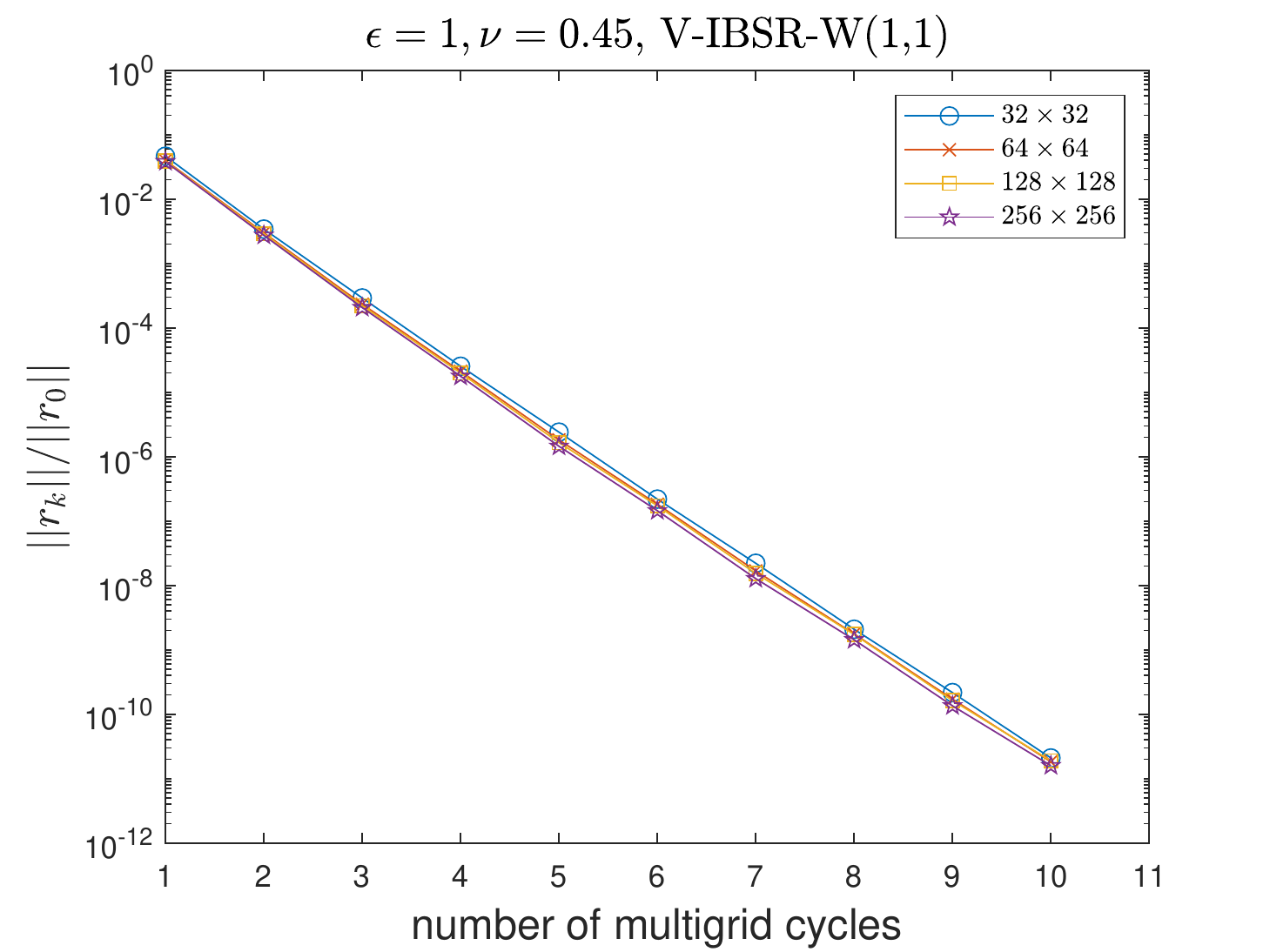}
\includegraphics[width=0.49\textwidth]{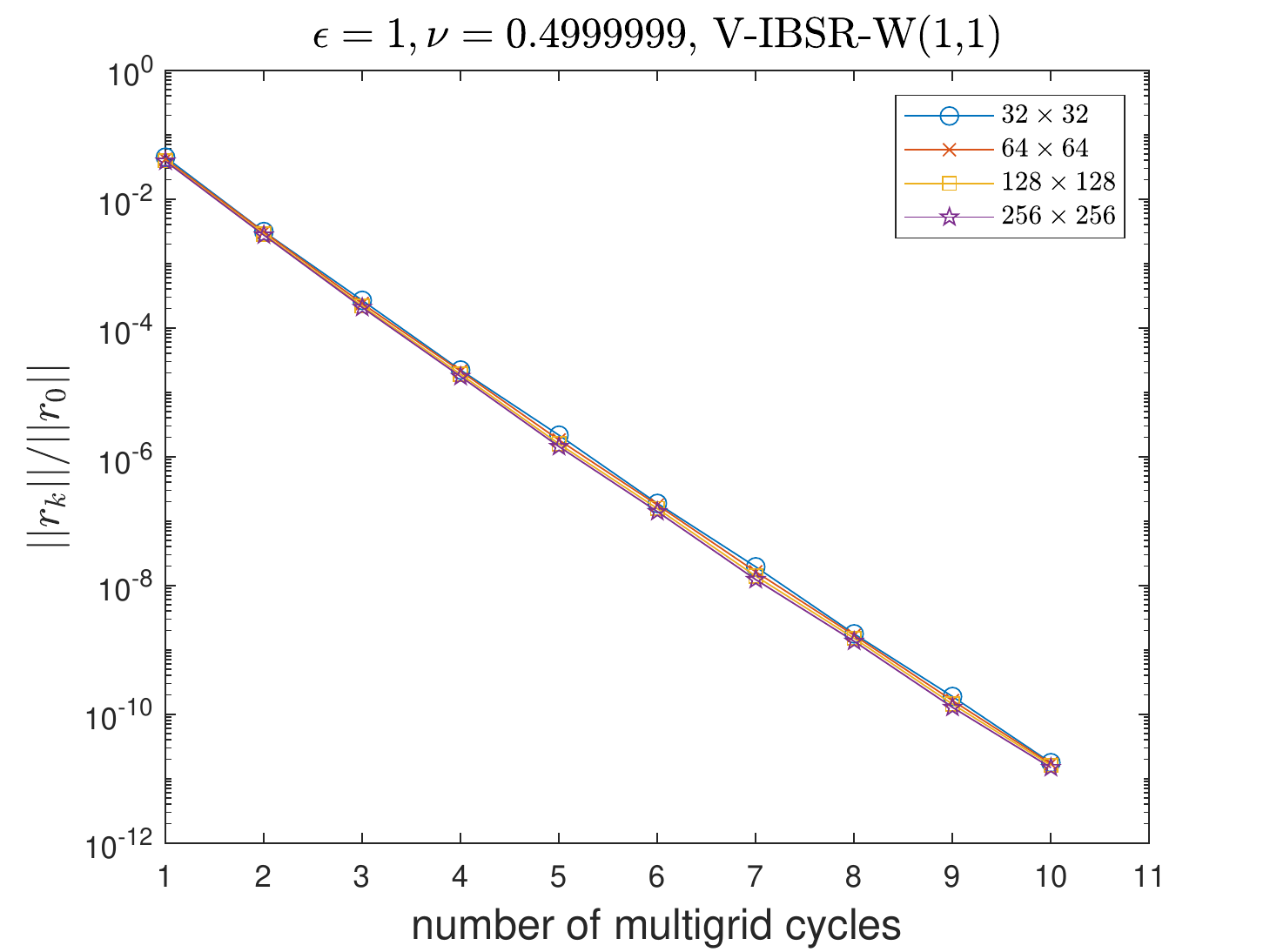}
 \caption{History of the convergence of  V-IBSR W(1,1)-multigrid method. Left: $\epsilon =1$ and $\nu=0.45$.  Right:  $\epsilon =1$ and $\nu=0.4999999$.}\label{fig: W-V-IBSR}
\end{figure}

 \subsubsection{V(1,1)-cycle results}
 In order to further explore multigrid method with these three relaxation schemes, we investigate V(1,1)-cycle multigrid for  $(\epsilon,\nu)=(1,0.45)$ and  $(\epsilon,\nu)=(1,0.4999999)$,  since V-cycle is cheaper than W-cycle.  For J-IBSR V(1,1)-cycle, we found that it takes more than one hundred iterations to achieve the stopping criterion for $h=1/32$, and it is divergent for $N=64, 128, 256$.  We tried  J-IBSR V(2,2)-cycle, and  it works,  see  Figure  \ref{fig: V-J-IBSR}.  However, with increasing the size of meshgrid, we see an increase in the number of iterations. Thus, J-IBSR V-cycle multigrid is not recommended.  For Q-IBSR V(1,1)-cycle in Figure \ref{fig: V-Q-IBSR}, and V-IBSR V(1,1)-cycle in Figure \ref{fig: V-V-IBSR},  the performances are the same as those of W(1,1)-cycle.  Again, we see the robustness  of Q-IBSR V(1,1)-cycle and V-IBSR V(1,1)-cycle with respect to $h$ and physical parameters, $(\epsilon,\nu)$. Since V(1,1)-cycle is cheaper than W(1,1)-cycle,  we recommend V(1,1)-cycle. 
 
 \begin{figure}[H]
\centering
\includegraphics[width=0.49\textwidth]{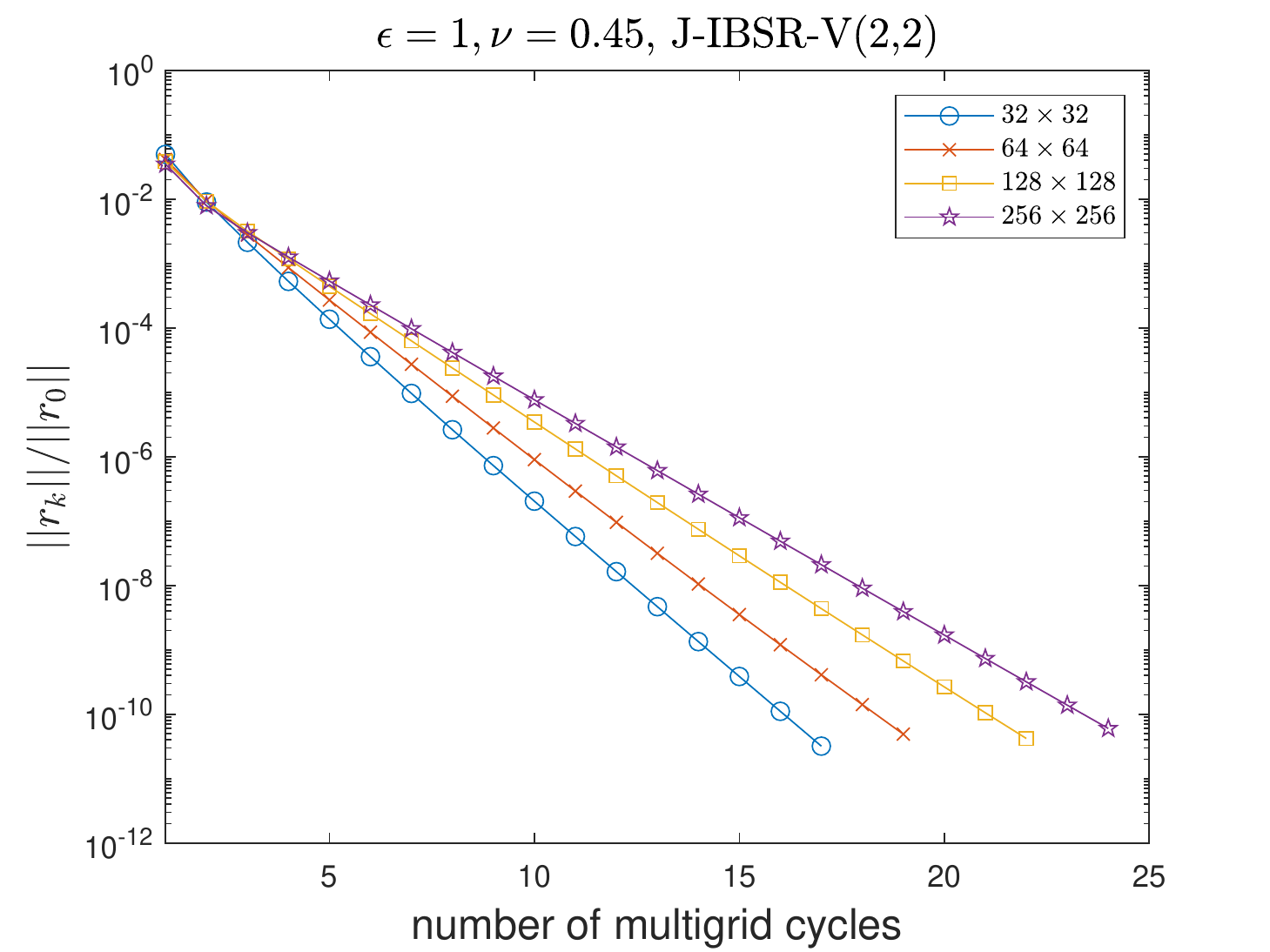}
\includegraphics[width=0.49\textwidth]{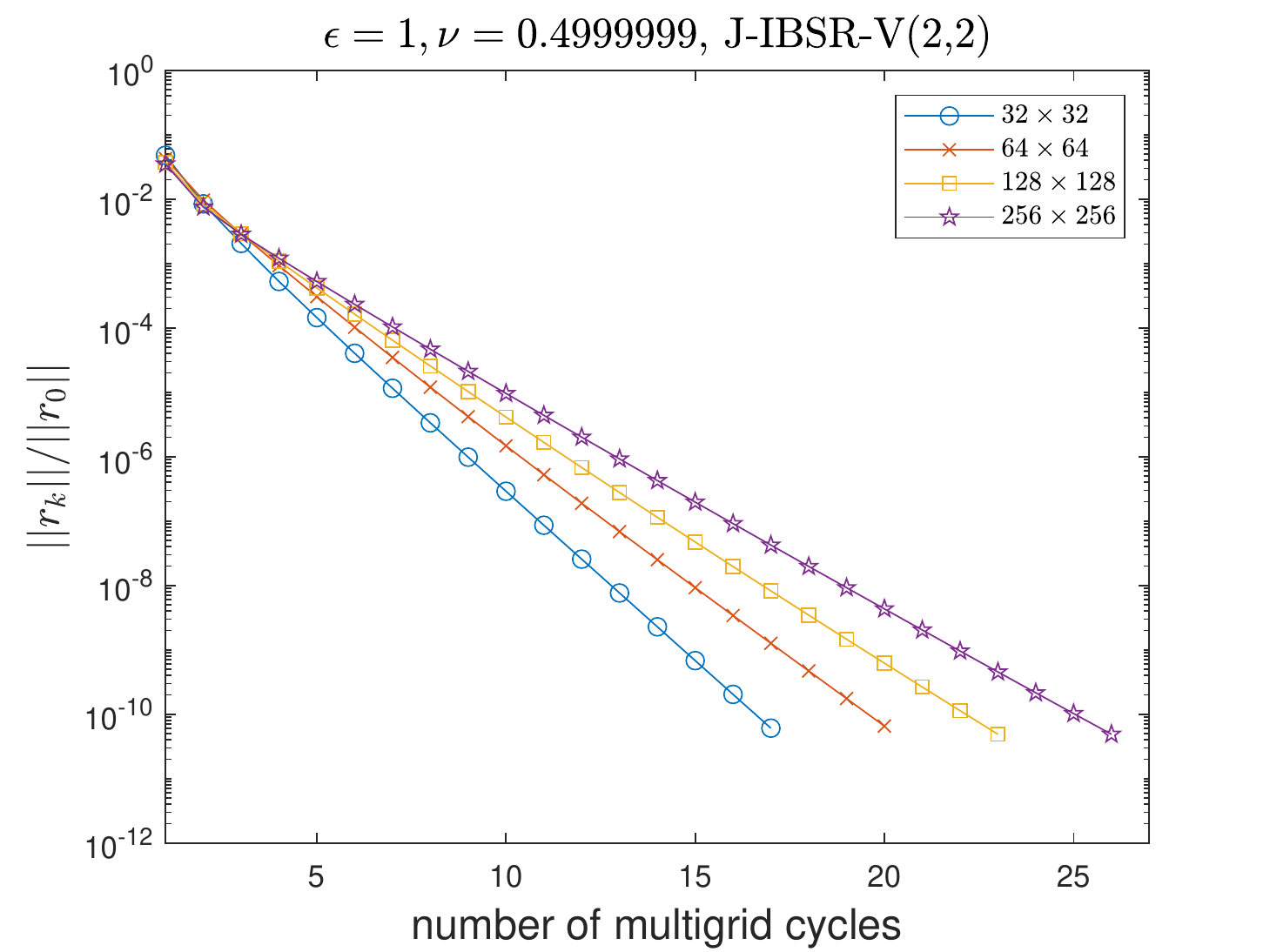}
 \caption{History of the convergence of  J-IBSR V(2,2)-multigrid method. Left:  $\epsilon =1$ and $\nu=0.45$. Right:  $\epsilon =1$ and $\nu=0.4999999$.}
 \label{fig: V-J-IBSR}
\end{figure}

 \begin{figure}[H]
\centering
\includegraphics[width=0.49\textwidth]{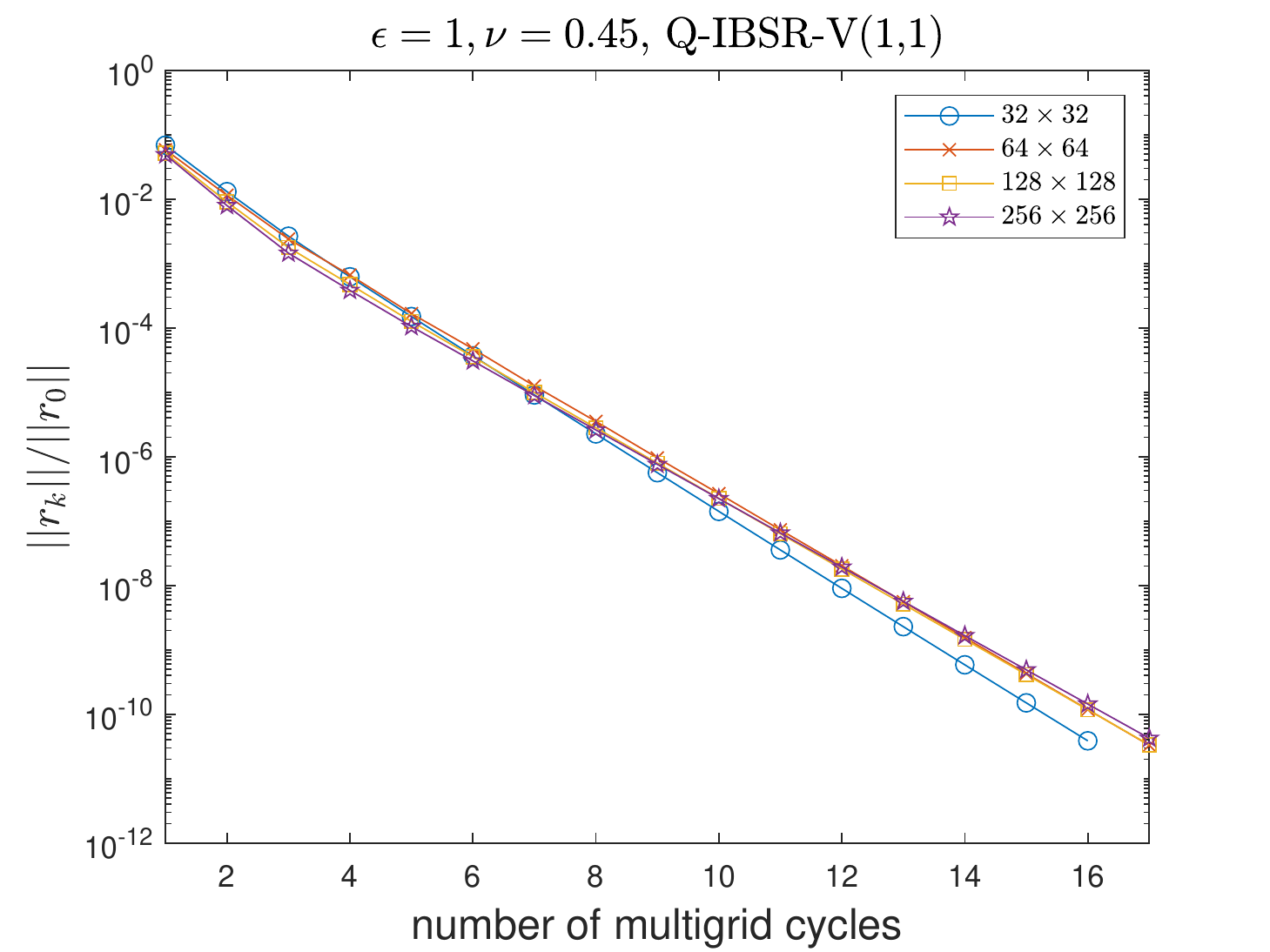}
\includegraphics[width=0.49\textwidth]{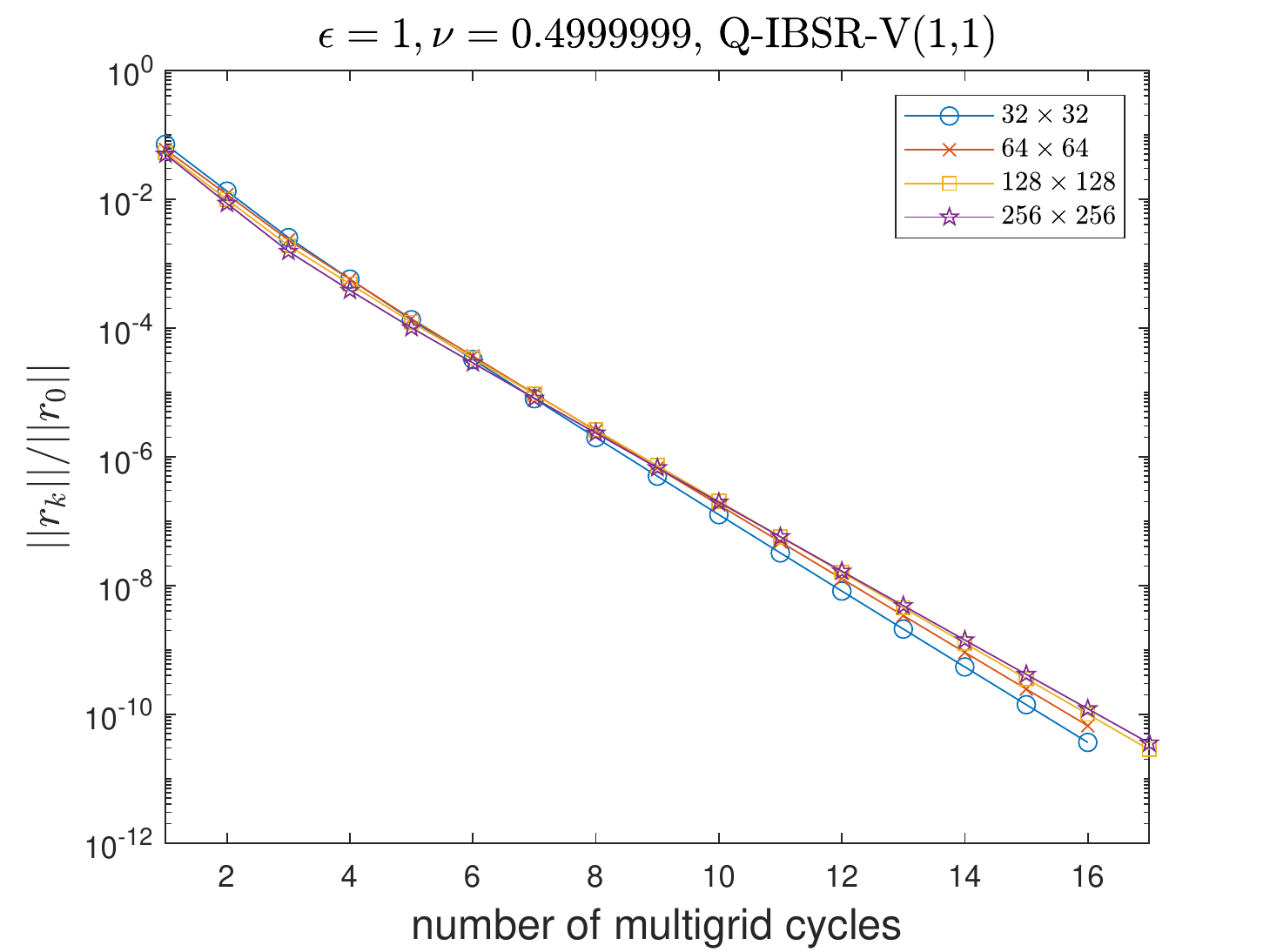}
 \caption{History of the convergence of  Q-IBSR V(1,1)-multigrid method. Left: $\epsilon =1$ and $\nu=0.45$.  Right:  $\epsilon =1$ and $\nu=0.4999999$.}
 \label{fig: V-Q-IBSR}
\end{figure}

\begin{figure}[H]
\centering
\includegraphics[width=0.49\textwidth]{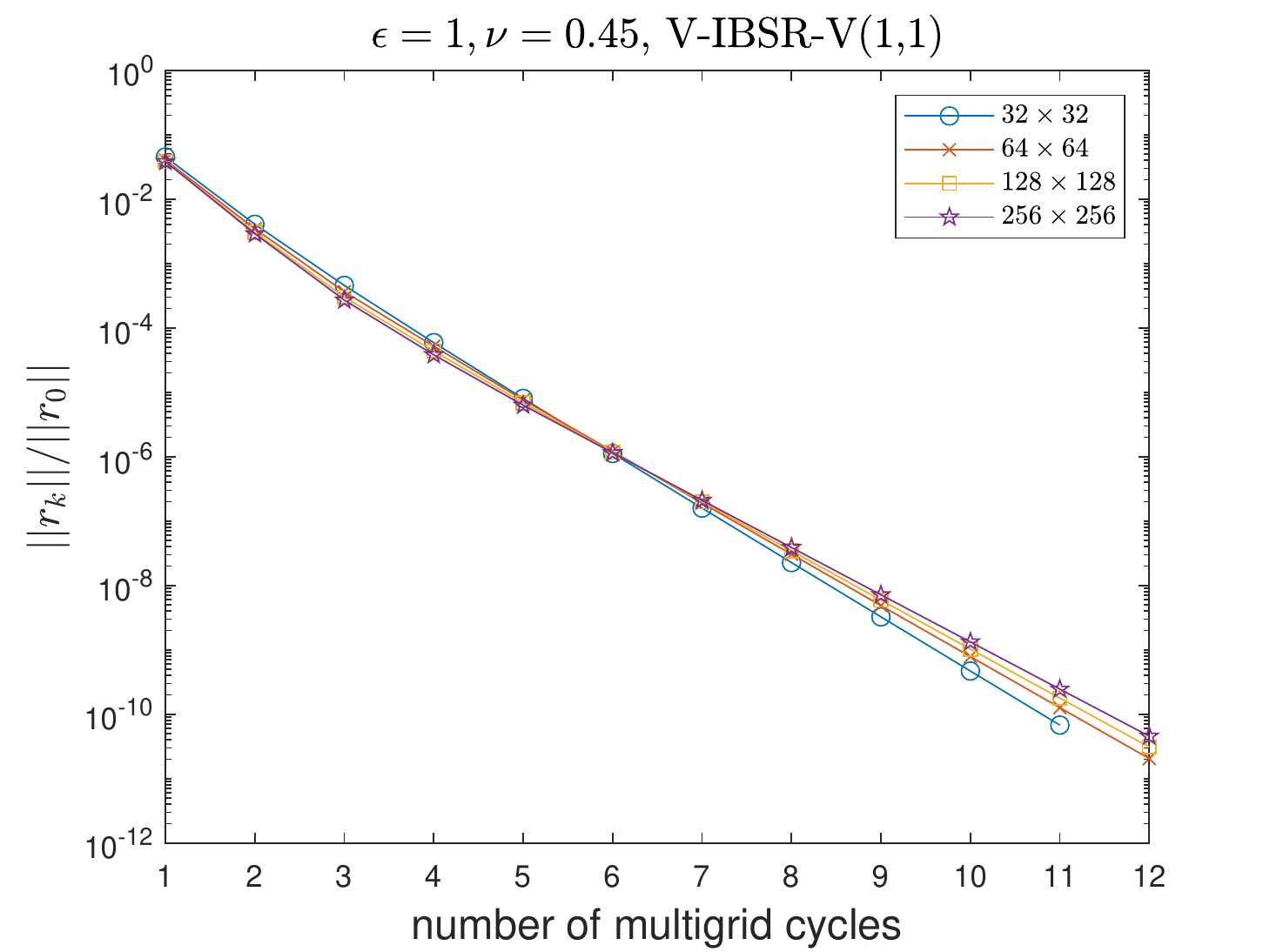}
\includegraphics[width=0.49\textwidth]{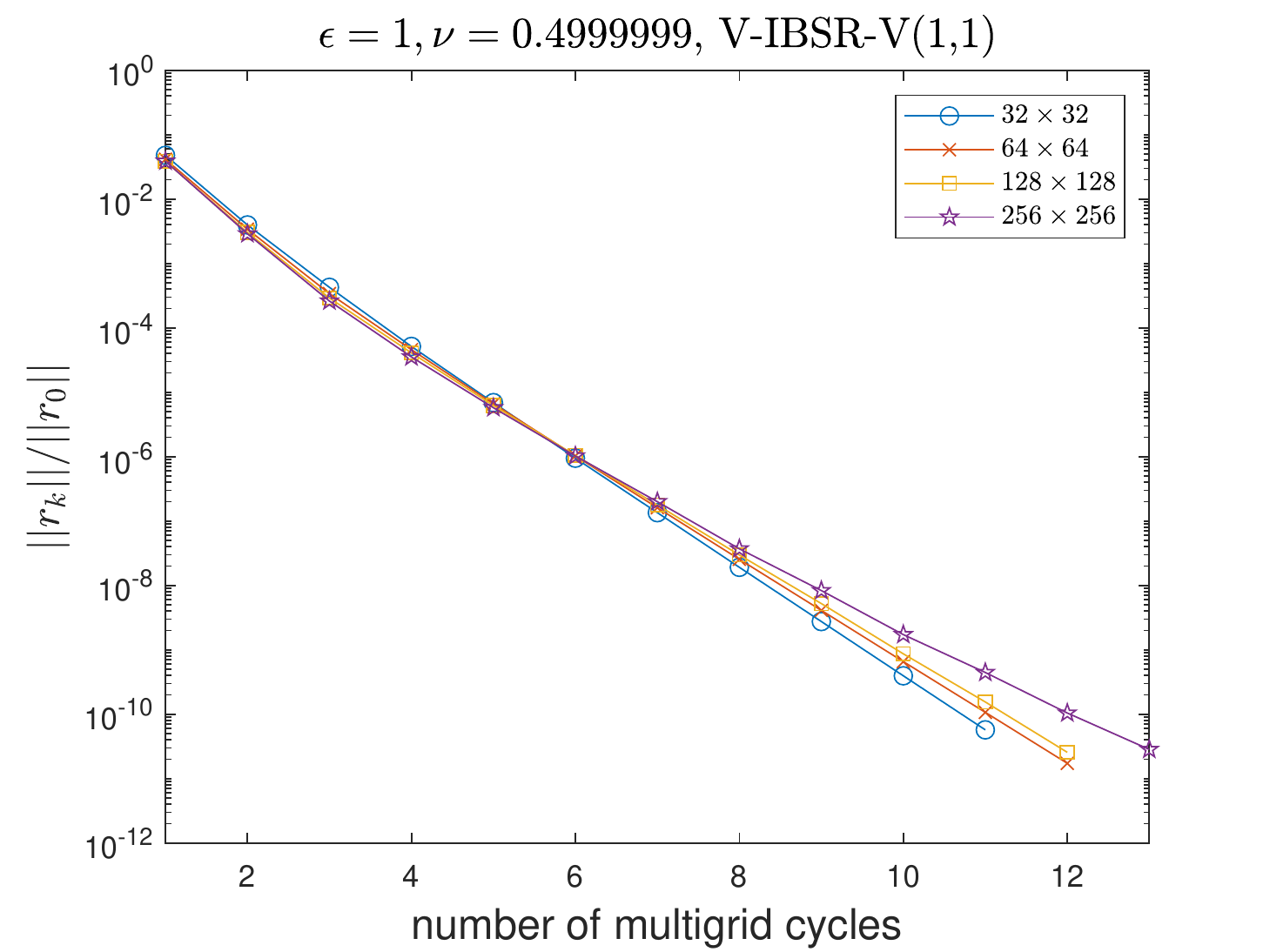}
 \caption{History of the convergence of  V-IBSR V(1,1)-multigrid method. Left: $\epsilon =1$ and $\nu=0.45$. Right:  $\epsilon =1$ and $\nu=0.4999999$.}\label{fig: V-V-IBSR}
\end{figure}

 \section{Conclusions}\label{sec:conclusion}
 
 We propose three block-structured Braess-Sarazin relaxation   schemes, named Jacobi-Braess-Sarazin relaxation, Mass-Braess-Sarazin relaxation, and 
 Vanka-Braess-Sarazin relaxation,  to solve  linear elasticity problems, where we consider a staggered finite difference discretization of a mixed form of the linear elasticity. To properly choose multigrid damping parameter and analyze multigrid performance, we consider local Fourier analysis.  From this analysis, we propose optimal damping parameters, and obtain highly efficient smoothing factors for these three relaxation schemes.  Our theoretical results show that the optimal smoothing factor is independent of  Lam\'{e} constants and meshsize,  and Vank-Braess-Sarazin relaxation is the best among  the three. 
 
In Braess-Sarazin relaxation, we need to solve a Schur complement system. To consider practical use, inexact versions of these relaxation schemes are proposed, where we apply at most three sweeps of Jacobi iteration to solve the Schur complement system.  Our numerical results show robustness to the Lam\'{e} constants.   We see that actual two-grid performance of Jacobi-Braess-Sarazin relaxation and Vanka-Braess-Sarazin relaxation match with our theoretical optimal smoothing factors, and there is a small degradation of Mass-Braess-Sarazin relaxation, which is reasonable since LFA prediction does not take into account  the influence of boundaries and of boundary conditions. For W(1,1)-cycle and V(1,1)-cycle multigrid, we see that Mass-Braess-Sarazin and  Vanka-Braess-Sarazin have the same performance as those of two-grid methods. For Jacobi-Braess-Sarazin relaxation, more pre- and post-smoothing steps are needed to obtain a good performance. Again, we see robustness of these methods to  Lam\'{e} constants and meshsize. Due to the smallest smoothing factor of Vanka-Braess-Sarazin  among the three,  V(1,1)-cycle Vanka-Braess-Sarazin is recommended.  It will be interesting to extend these relaxation schemes to MAC scheme for linear elasticity problems discretized on non-uniform grids, which will be our future work.
\bibliographystyle{plain}
\bibliography{linear_elasticitybib}


\end{document}